\documentclass[10pt]{amsart}

\usepackage[all]{xy}

\usepackage[normalem]{ulem}

\usepackage{enumitem}

\usepackage{amsmath, amssymb, amsthm, latexsym, amscd}

\DeclareMathOperator{\HH}{H}
\DeclareMathOperator{\IHH}{IH}
\DeclareMathOperator{\IC}{IC}
\DeclareMathOperator{\II}{I}
\DeclareMathOperator{\Ext}{Ext}
\DeclareMathOperator{\HHom}{Hom}

\newtheorem{theorem}[subsection]{Theorem}
\newtheorem{prop}[subsection]{Proposition}
\newtheorem{lemma}[subsection]{Lemma}
\newtheorem{corollary}[subsection]{Corollary}

\newtheorem{conjecture}[subsection]{Conjecture}  

\theoremstyle{definition}
\newtheorem{define}[subsection]{Definition}
\newtheorem{notation}[subsection]{Notation}
\newtheorem{example}[subsection]{Example}

\newtheorem{remark}[subsection]{Remark}

\title{On the hyperplane conjecture for periods of Calabi-Yau hypersurfaces in $\mathbb P^n$}

\author{Bong H. Lian}
\address{Department of Mathematics, Brandeis University, Waltham MA 02454, U.S.A.}
\email{lian@brandeis.edu}

\author{Minxian Zhu}
\address{Yau Mathematical Sciences Center, Tsinghua University, Beijing 100084, P.R.China}
\email{mxzhu@math.tsinghua.edu.cn}

\begin{document}
\maketitle

\begin{abstract}
In [HLY1], Hosono, Lian, and Yau posed a conjecture characterizing the set of solutions to certain Gelfand-Kapranov-Zelevinsky hypergeometric equations which are realized as periods of Calabi-Yau hypersurfaces in a Gorenstein Fano toric variety $X$. We prove this conjecture in the case where $X$ is a complex projective space.
\end{abstract}

\section{Introduction}

Originating from physics, 
mirror symmetry drew the attention of mathematicians
when it was applied to derive astonishing predictions about the number of rational curves on the quintic threefold using the periods of another family of ``mirror" Calabi-Yau threefolds ([CdGP]).
This miracle led physicists and mathematicians to search for more such ``mirror pairs" of Calabi-Yau manifolds in order to better understand this duality phenomenon.

Soon after,
Batyrev constructed many such examples of mirror pairs in toric geometry
using the polar duality of lattice polytopes ([B2]).
More precisely, 
let $\triangle, \triangle^\vee$ be a pair of reflexive lattice polytopes, 
with $X = \mathbb P_\triangle$, 
$X^\vee = \mathbb P_{\triangle^\vee}$
being the corresponding projective toric varieties. 
For the purpose of this work, 
we assume that both $\triangle$ and $\triangle^\vee$ 
admit regular projective triangulations
(see Section 2), 
meaning that both $X^\vee$ and $X$ admit projective crepant toric resolutions.
Let $Y$ and $Y^\vee$ be $\triangle$-regular and $\triangle^\vee$-regular 
Calabi-Yau hypersurfaces in $X$ and $X^\vee$ respectively, 
with $\widetilde Y$ and $\widetilde{Y^\vee}$ being their resolutions
induced by the resolutions of the ambient spaces.
Batyrev conjectured that the two families of Calabi-Yau hypersurfaces 
$\widetilde Y$ and $\widetilde{Y^\vee}$ form a mirror pair.
The duality of Hodge numbers is one piece of supporting evidence
([BBo2]).

Batyrev observed that 
the period integrals of the $\widetilde Y$-family 
are solutions of the $\mathcal A$-hypergeometric system
introduced by Gelfand, Kapranov, and Zelevinsky, 
where $\mathcal A$ is the set of lattice points in $1 \times \triangle$, 
and with parameter $\beta=(1, 0)$
\footnote{We use the convention in [A].}
([B1], [GKZ]).
With applications to mirror symmetry in mind, 
Hosono, Lian, and Yau studied this semi-nonresonant 
\footnote{in the sense of [A]}
GKZ-system.
They proposed a certain compactification of the moduli space of $\widetilde Y$, 
that is, the toric variety associated to the secondary fan of $\triangle$.
On this compactification exist so-called maximal degeneracy points 
where all periods become singular except one 
([HLY2]).
These maximal degeneracy points correspond to the 
regular projective triangulations of $\triangle$.
Let $\mathcal T$ be such a triangulation, 
and let $\widetilde{X^\vee_{\mathcal T}}$ be the smooth projective toric variety associated to the refined fan $\Sigma_{\triangle, \mathcal T}$.
Hosono, Lian, and Yau constructed explicit solutions of the GKZ-system near this point, 
and they could be expressed as a function $B_{X, \mathcal T}(a)$ on the parameter space with values in the cohomology ring of $\widetilde{X^\vee_{\mathcal T}}$
([HLY1]).

However, not all solutions of the GKZ-system are periods of $\widetilde Y$. 
Even if we enlarge the GKZ-system by incorporating the infinitesimal action of the full automorphism group of $\widetilde X$ instead of just the torus action, 
the extended system could still have solutions which are not periods.
Hosono, Lian, and Yau computed many examples.
Based on empirical evidence, 
they conjectured that
the cup product of $B_{X, \mathcal T}(a)$ with the Calabi-Yau class 
$[\widetilde{Y^\vee_{\mathcal T}}]$ give precisely the set of periods for $\widetilde Y$
near the maximal degeneracy point determined by $\mathcal T$.

The main purpose of this paper is to prove the HLY conjecture for $X = \mathbb P^n$.
We will show in the appendix that the reflexive polytope $\triangle$ which defines $\mathbb P^n$ 
admits a regular projective triangulation.
It is obvious that the dual polytope $\triangle^\vee$ satisfies the same property.
Our proof relies on the work of [BHLSY], 
where it was proved that 
the period integrals of Calabi-Yau hypersurfaces in $\mathbb P^n$
are precisely the solutions of the extended GKZ-system, 
which in this case is obtained from the GKZ-system by adding differential operators corresponding to the root vectors of the simple Lie algebra $sl(n+1, \mathbb C)$.
First, 
we show that the coefficient functions in 
$B_{\mathbb P^n, \mathcal T}(a) \!\!\smile\!\! [\widetilde{Y^\vee_{\mathcal T}}]$
satisfy the extended GKZ-system, 
hence they are periods by [BHLSY].
Second, 
we show that 
the dimension of the space of coefficient functions in 
$B_{\mathbb P^n, \mathcal T}(a) \!\!\smile\!\! [\widetilde{Y^\vee_{\mathcal T}}]$, 
or equivalently the rank of the cupping map
$\!\smile\!\![\widetilde{Y^\vee_{\mathcal T}}]: \HH^\bullet(\widetilde{X^\vee_{\mathcal T}}, \mathbb Q)
\to \HH^\bullet(\widetilde{X^\vee_{\mathcal T}}, \mathbb Q)$, 
is equal to the rank of the period sheaf, 
which, in this case, 
is the dimension of the vanishing cohomology of a Calabi-Yau hypersurface $Y$ in $\mathbb P^n$.
Let $\pi^\vee: \widetilde{X^\vee_{\mathcal T}} \to X^\vee$ 
be the crepant resolution induced by $\mathcal T$.
Since $[\widetilde{Y^\vee_{\mathcal T}}] = (\pi^\vee)^* [Y^\vee]$, 
by a formal argument, 
the above rank can be computed on $X^\vee$; 
more precisely it is equal to the rank of 
$\!\smile\!\! [Y^\vee]: 
\mathbb H^\bullet(X^\vee, R \pi^\vee_* \underline{\mathbb Q}_{\widetilde{X^\vee_{\mathcal T}}}) \to 
\mathbb H^\bullet(X^\vee, R \pi^\vee_* \underline{\mathbb Q}_{\widetilde{X^\vee_{\mathcal T}}})$.
A toric version of the Beilinson-Bernstein-Deligne-Gabber decomposition theorem implies that
the derived pushforward 
$R \pi^\vee_* \underline{\mathbb Q}_{\widetilde{X^\vee_{\mathcal T}}}$
is isomorphic, 
in the constructible bounded derived category,
to a direct sum of intersection complexes of irreducible toric subvarieties of $X^\vee$ 
up to cohomological shifts. 
The multiplicities of the summands in the decomposition are combinatorial invariants
([dCMM]). 
Since $Y^\vee$ is ample, 
the hard Lefschetz theorem for intersection cohomology and the combinatorial tools in toric geometry allow us to compute the rank on $X^\vee$.

Though the conjecture is only proved for the complex projective spaces, 
some of the results we prove hold in general.
For example, 
we show that the coefficient functions in 
$B_{X, \mathcal T}(a) \!\!\smile\!\! [\widetilde{Y^\vee_{\mathcal T}}]$
are always solutions of the extended GKZ-system.
The decomposition theorem could be used to show that the rank of 
$\!\smile\!\! [\widetilde{Y^\vee_{\mathcal T}}]: \HH^\bullet(\widetilde{X^\vee_{\mathcal T}}, \mathbb Q)
\to \HH^\bullet(\widetilde{X^\vee_{\mathcal T}}, \mathbb Q)$
is, in general, 
equal to the dimension of the vanishing intersection cohomology of a $\triangle$-regular Calabi-Yau hypersurface $Y$ in $X$. 
Part of this calculation is implicit in [BoM], 
where the authors used different tools to study the cohomology ring of 
$\widetilde{Y^\vee_{\mathcal T}}$ and its mirror.
By a result of Mavlyutov, 
the rank above is, in fact, 
the dimension of the toric part of the cohomology of 
$\widetilde{Y^\vee_{\mathcal T}}$. 
We make a few remarks in Section 6 about how our work relates to [BoM].

The paper is organized as follows.
In Section 2, we briefly recall Batyrev's dual reflexive polytope construction and 
the semi-nonresonant GKZ-hypergeometric system relevant to our consideration. 
After we describe the explicit solutions of the GKZ-system at the large complex structure limit, 
we state the hyperplane conjecture and also our main result. 
Section 3 and 4 are devoted to the proof of the main theorem. 
In Section 3, 
we derive the extra differential operators that make up the extended GKZ-system
and show that the coefficient functions in 
$B_{X, \mathcal T}(a) \!\!\smile\!\! [\widetilde{Y^\vee_{\mathcal T}}]$
are their solutions.
In Section 4,
we recall some basic properties of the intersection cohomology 
of a complex algebraic variety and 
(a toric version of) the decomposition theorem. 
We show how to use the decomposition theorem 
to compute the number of linearly independent
period integrals predicted by the HLY conjecture
in the case of projective spaces.
In Section 5, 
we examine more closely the numbers of period integrals of Calabi-Yau hypersurfaces 
in a projective space 
with various leading terms
$1, \log(a_i), \log(a_i)\log(a_j), \ldots$
at the large complex structure limit, 
and match them with some combinatorially defined integers.
Section 6 is a generalization of Section 5
where we use the decomposition theorem to derive a general formula for the rank of 
$\smile\!\! [\widetilde{Y^\vee_{\mathcal T}}]: \HH^\bullet(\widetilde{X^\vee_{\mathcal T}}, \mathbb Q)
\to \HH^\bullet(\widetilde{X^\vee_{\mathcal T}}, \mathbb Q)$.
We also remark how our approach is related to the previous work of [BoM].
In the appendix, 
we show that the reflexive polytopes which define the complex projective spaces
admit regular projective triangulations.

{\bf Acknowledgement.} 
M. Zhu would like to thank Eduard Looijenga for valuable discussions and a course he gave at YMSC on topology of algebraic varieties, Shenghao Sun for his patience and generosity with answering questions on the decomposition theorem and perverse sheaves, and Lev Borisov for bringing the work of [BoM] to her attention. 
M. Zhu is partially supported by China NSF grant 11201255, China NSF grant 11531007, 
and a fellowship from the China Thousand Talents Program.
B.H. Lian would like to thank S. Hosono, A. Huang and S.-T. Yau for helpful discussions over the course of this work. 
He also thanks the Tsinghua YMSC for hospitality during his recent visits there.
His research is partially supported by an NSF FRG grant MS-1564405 and a Simons Foundation Collaboration Grant (2015-2019).

\section{Solutions of the GKZ-hypergeometric equations at LCSL}

Let $\triangle$ be a reflexive lattice polytope in $\mathbb R^n$, 
with dual polytope $\triangle^\vee$.
Let $\Sigma_\triangle$ be the fan consisting of cones over the faces of $\triangle$.
Equivalently, 
$\Sigma_\triangle$ is the normal fan of the dual polytope $\triangle^\vee$.
Define $\Sigma_{\triangle^\vee}$ similarly.
Let $X = \mathbb P_\triangle$ and $X^\vee = \mathbb P_{\triangle^\vee}$ 
be the projective toric varieties associated to the fans 
$\Sigma_{\triangle^\vee}$ and $\Sigma_\triangle$ respectively.
Both $X$ and $X^\vee$ are Gorenstein Fano toric varieties. 
The lattice points in $\triangle$ and $\triangle^\vee$ 
form bases for the global sections of the anticanonical line bundles 
on $X$ and $X^\vee$ respectively. 
That is, 
$$
\HH^0(X, -\omega_X) = \mathbb C[\triangle \cap \mathbb Z^n], \quad 
\HH^0(X^\vee, -\omega_{X^\vee}) = \mathbb C[\triangle^\vee \cap \mathbb Z^n].
$$
Let $Y$ be a Calabi-Yau hypersurface in $X$, 
defined by a section of $-\omega_X$, 
intersecting each toric orbit of $X$ transversely.
This is Batyrev's $\triangle$-regularity condition [B2]. 
It implies that the singularities of $Y$ are induced only by the singularities of $X$. 
Therefore, 
a resolution of singularities of $X$ yields a resolution of singularities 
of all $\triangle$-regular hypersurfaces.
Similarly, 
we have $\triangle^\vee$-regular Calabi-Yau hypersurfaces $Y^\vee$ in $X^\vee$. 
Let $\widetilde Y$ be a maximal projective crepant partial desingularization 
(MPCP-desingularization) of $Y$, 
similarly let $\widetilde{Y^\vee}$ be a MPCP-desingularization of $Y^\vee$. 
Batyrev conjectured that the Calabi-Yau families 
$\widetilde Y$ and $\widetilde{Y^\vee}$ 
are mirror to each other ([B2]).
This construction was later generalized by Batyrev and Borisov to Calabi-Yau complete intersections in $X$ and $X^\vee$ ([BBo1]).
They proved that the string-theoretic Hodge numbers ([BD]) of the conjectured mirror families satisfy the expected duality ([BBo2]).

Let $\mathcal T$ be a triangulation of $\triangle$ with lattice points as vertices such that the origin is a vertex of every maximal dimensional simplex in $\mathcal T$ and the resulting subdivision of the fan $\Sigma_\triangle$ gives rise to a smooth projective toric variety $\widetilde{X^\vee}$. 
We call $\mathcal T$ a \emph{regular projective triangulation} 
\footnote{Regular means unimodular, i.e. the simplices all have normalized volume 1; 
projective means regular in the sense of [GKZ], [HLY1], i.e.
there exists a strictly-convex $\mathcal T$-piecewise linear function on $\triangle$;
we require additionally that $0$ is the vertex of every maximal dimensional simplex in $\mathcal T$.}
of $\triangle$.
Such a triangulation may not always exist.
However, 
if it does exist, 
then $\widetilde{X^\vee}$ is a projective crepant resolution of $X^\vee$.
Let $\mathcal T^\vee$ be a regular projective triangulation of $\triangle^\vee$, 
and $\widetilde X$ be the corresponding projective crepant resolution of $X$.
Throughout the paper, 
we assume such triangulations exist for $\triangle$ and $\triangle^\vee$.
Note that we still have 
\begin{eqnarray}
\HH^0(\widetilde X, -\omega_{\widetilde X}) & = & \HH^0(X, -\omega_X) \quad = \quad \mathbb C[\triangle \cap \mathbb Z^n] 
\nonumber \\
\HH^0(\widetilde{X^\vee}, -\omega_{\widetilde{X^\vee}}) & = & \HH^0(X^\vee, -\omega_{X^\vee})  \quad = \quad \mathbb C[ \triangle^\vee \cap \mathbb Z^n].
\nonumber
\end{eqnarray}
Let $f(t) = \sum_{v_i \in \triangle \cap \mathbb Z^n} a_i t^{v_i}$ 
be a section of $-\omega_{\widetilde X}$
defining a smooth Calabi-Yau hypersurface $\widetilde Y_f$ in $\widetilde X$.
Then 
$$
\frac{\omega}{f} = \frac{ \frac{d t_1}{t_1} \wedge \frac{d t_2}{t_2} \wedge \cdots \wedge \frac{d t_n}{t_n}}{f(t_1, \ldots, t_n)}
$$
is a meromorphic $n$-form on $\widetilde X$ 
with a simple pole along $\widetilde Y_f$.
The Poincar\'{e} residue of $\frac{\omega}{f}$ then defines a 
nowhere-vanishing holomorphic $(n-1)$-form on $\widetilde Y_f$, 
which we denote by $\omega_f$.
Let $f$ vary and consider the family $\widetilde{\mathcal Y}: f \mapsto \widetilde Y_f$.
The cohomology groups $\HH^\bullet(\widetilde Y_f, \mathbb C)$ form a local system
with respect to the Gauss-Manin connection. 
The locally-defined functions 
$$
\Pi_\gamma(a): f \mapsto \int_{\gamma \in \HH_{n-1}(\widetilde Y_f, \mathbb C)} \omega_f, 
\quad \quad a = (a_i)_i
$$
are called \emph{period integrals}.
The subsheaf they generate over $\mathbb C$ is called the \emph{period sheaf}.

Period integrals satisfy the so-called Gelfand-Kapranov-Zelevinsky (GKZ) hypergeometric equations ([B1]).
Specifically, 
let $\mathcal A = \triangle \cap \mathbb Z^n = \{ v_0=0, v_1, \ldots, v_p \}$ 
be the set of integral points in $\triangle$, and 
let $L$ be the lattice of affine linear relations among them, i.e. 
\begin{eqnarray}
L & = & \{ (\ell_0, \ell_1, \ldots, \ell_p) \in \mathbb Z^{p+1}: \,\, \ell_0 v_0 + \ell_1 v_1 + \cdots + \ell_p v_p = 0 \text{    and    } \ell_0 + \ell_1 + \cdots + \ell_p = 0 \}
\nonumber
\end{eqnarray}
Then the period integrals satisfy the following GKZ-system of differential equations
\begin{eqnarray}
\left( \Pi_{\ell_i >0} \left( \frac{\partial}{\partial a_i} \right)^{\ell_i} - \Pi_{\ell_j <0} \left( \frac{\partial}{\partial a_j} \right)^{-\ell_j} \right) \,\, \Pi_\cdot (a) & = & 0 \quad ( \forall \ell \in L)
\label{D_l}
\\
\left( \sum_{i=1}^p v_i^j a_i \frac{\partial}{\partial a_i} \right) \,\, \Pi_\cdot (a) & = & 0 \quad (j=1, \ldots, n)
\label{t-invariance} 
\\
\left( \sum_{i=0}^p a_i \frac{\partial}{\partial a_i} + 1 \right) \,\, \Pi_\cdot (a) & = & 0
\label{Euler}
\end{eqnarray}
where $v_i = (v_i^1, \ldots, v_i^n) \in \mathbb Z^n$.

Not all solutions to the GKZ-equations are periods.
Note that the second equation (\ref{t-invariance}) 
is the $\mathfrak t$-invariance condition with respect to the action of the torus $T$ on 
$\widetilde X$ and $\HH^0(\widetilde X, -\omega_{\widetilde X})$.
One may add the invariance operators with respect to the infinitesimal action of the full automorphism group of $\widetilde X$; 
the enlarged system of differential equations is called the \emph{extended GKZ-system}
(see Section 3).
The period integrals are also solutions of the extended system.
However, the extended system is still not complete, in the sense that 
it could have solutions which are not periods.

After a close study of examples, 
Hosono, Lian, and Yau formulated a conjecture describing 
the period solutions of the GKZ-system
near the large complex structure limit (LCSL), 
where the Calabi-Yau hypersurface $\widetilde Y$ 
degenerates into the union of $T$-invariant divisors of $\widetilde X$, 
each with multiplicity one.
In fact, 
explicit solutions of the GKZ-system at the LCSL were constructed in [HLY1]. 
We will describe these solutions below and then state the conjecture.

Let $\mathcal T$ be a regular projective triangulation of $\triangle$, and 
$\Sigma_{\triangle, \mathcal T}$ be the corresponding subdivision of the fan $\Sigma_\triangle$.
Denote the set of $k$-dimensional cones in $\Sigma_{\triangle, \mathcal T}$ by 
$\Sigma_{\triangle, \mathcal T}(k)$.
Let $\widetilde{X^\vee} = \widetilde{X^\vee_{\mathcal T}}$
be the toric variety defined by the fan $\Sigma_{\triangle, \mathcal T}$.
Then $\pi^\vee: \widetilde{X^\vee} \to X^\vee$ 
is a projective crepant resolution of $X^\vee$.
Set
$\widetilde{Y^\vee} = \widetilde{Y^\vee_{\mathcal T}} = {\pi^\vee}^{-1} (Y^\vee)$. 
The Picard group of $\widetilde{X^\vee}$ fits into a short exact sequence
\begin{eqnarray}
& 0 \to \mathbb Z^n \to \mathbb Z^p \to \text{Pic} (\widetilde{X^\vee}) \to 0.& 
\nonumber
\end{eqnarray}
Indeed, 
$\mathbb Z^p$ is the free abelian group generated by the $T$-invariant Weil divisors 
$D_1, \ldots, D_p$ of $\widetilde{X^\vee}$.
Note that they correspond to the nonzero integral points $v_1, \ldots, v_p$ of $\triangle$.
The map $\mathbb Z^n \to \mathbb Z^p$ sends the $j$-th standard basis vector of $\mathbb Z^n$ to $\sum_{i=1}^p v_i^j D_i$.
It is convenient to enlarge the above short exact sequence to
\begin{eqnarray}
& 0 \to \mathbb Z^{n+1} \to \mathbb Z^{p+1} \to \text{Pic} (\widetilde{X^\vee}) \to 0. & 
\label{Picard}
\end{eqnarray}
This is understood as follows:
$\mathbb Z^{p+1}$ is the free abelian group generated by $D_0, D_1, \ldots, D_p$, where 
$D_0$ is an artificial divisor we add for $v_0=0 \in \triangle$, 
subject to the linear relation 
$D_0 + D_1 + \cdots + D_p = 0$. 
Thus, 
$D_0 = -\sum_{i=1}^p D_i$ is the canonical divisor of $\widetilde{X^\vee}$.

As a group, 
$\text{Pic} (\widetilde{X^\vee})$ is free of rank $p-n$, 
and isomorphic to $\HH^2(\widetilde{X^\vee}, \mathbb Z)$.
The lattice $L$ is naturally identified with 
$\text{Hom}_{\mathbb Z} (\text{Pic} (\widetilde{X^\vee}), \mathbb Z)$
via
$$\ell = (\ell_0, \ell_1, \ldots, \ell_p) \quad  \mapsto \quad (D_i \mapsto \ell_i, 0 \leq i \leq p).
$$
Let $\mathcal K \subset \HH^2(\widetilde{X^\vee}, \mathbb Z) \otimes_{\mathbb Z} \mathbb R$ 
be the K\"{a}hler cone of $\widetilde{X^\vee}$, 
and $\mathcal M = \mathcal K^\vee \subset L \otimes_{\mathbb Z} \mathbb R$ 
be the Mori cone.
It is known that $\mathcal M$ is generated by the primitive relations
of $\Sigma_{\triangle, \mathcal T}$
([CLS], Thereom 6.4.11).

\begin{define}
A subset $P = \{ v_{i_1}, v_{i_2}, \ldots, v_{i_k} \} \subset \mathcal A \backslash \{ 0 \} = \Sigma_{\triangle, \mathcal T} (1)$
is called a \emph{primitive} collection if $P$ is not contained in $\sigma(1)$ for any $\sigma \in \Sigma_{\triangle, \mathcal T}$, 
but any proper subset is.
\end{define}

\begin{lemma} 
[{[HLY2]}]
\label{primitive}
If $P = \{v_{i_1}, \ldots, v_{i_k} \}$ is a primitive collection, 
then there exists a unique simplex $\theta$ of $\mathcal T$ with vertex set
$\{ v_{j_1}, \ldots, v_{j_l}\} \subset \mathcal A\backslash \{ v_{i_1}, \ldots, v_{i_k} \}$ 
such that 
\begin{eqnarray}
v_{i_1} + \cdots + v_{i_k} & = &  c_1 v_{j_1} + \cdots + c_l v_{j_l}
\nonumber
\end{eqnarray}
for some $c_1, \ldots, c_l \in \mathbb Z_{>0}$ and $ c_1 + \cdots + c_l = k$.
\end{lemma}

This defines an element $\ell = (\ell_i)_i$ of $L$ with 
$\ell_{i_1} = \cdots = \ell_{i_k} = 1$ and 
$\ell_{j_m} = - c_m$ for $1 \leq m \leq l$.
We call such an $\ell$ 
a \emph{primitive relation} of the triangulation $\mathcal T$.
In particular, 
the coefficient of $v_0$ in any primitive relation is nonpositive.

\begin{proof}
Since $\triangle$ is convex, the point 
$\frac{v_{i_1} + \cdots + v_{i_k}}{k}$ lies in $\triangle$, and therefore lies in the relative interior of a unique simplex $\theta$ of $\mathcal T$ with vertices $\{v_{j_1}, \ldots, v_{j_l} \}$.
It follows that there exist $q_1, \ldots, q_l \in \mathbb Q_{>0}$ such that 
$\frac{v_{i_1} + \cdots + v_{i_k}}{k} = q_1 v_{j_1} + \cdots + q_l v_{j_l}$ 
and $q_1 + \cdots + q_l = 1$.
Multiplying with $k$, 
we get 
$v_{i_1} + \cdots + v_{i_k} = kq_1 v_{j_1} + \cdots + kq_l v_{j_l} =
k v_{j_1} + k q_2 (v_{j_2} - v_{j_1}) + \cdots + k q_l (v_{j_l} - v_{j_1})$.
Since $v_{i_1}, \ldots, v_{i_k}, v_{j_1}$ are integral, 
so is $k q_2 (v_{j_2} - v_{j_1}) + \cdots + k q_l (v_{j_l} - v_{j_1})$.
But the simplex $\theta$ is regular, 
which implies that the vectors $v_{j_2} - v_{j_1}, \ldots, v_{j_l} - v_{j_1}$ 
form a $\mathbb Z$-basis of the sublattice 
$\text{Span}_{\mathbb Q} (v_{j_2} - v_{j_1}, \ldots, v_{j_l} - v_{j_1}) \cap \mathbb Z^n$.
Hence,
$k q_2, \ldots, k q_l \in \mathbb Z$ and therefore $k q_1 = k - kq_2 - \cdots - k q_l \in \mathbb Z$.
Set $c_m = k q_m$, $1 \leq m \leq l$.

If $\{v_{i_1}, \ldots, v_{i_k} \} \cap \{ v_{j_1}, \ldots, v_{j_l} \} \neq \emptyset$,
say $v_{j_1} = v_{i_1}$, 
then, 
since $P$ is a primitive collection, 
$\{ v_{i_2}, \ldots, v_{i_k} \}$ is the vertex set of a simplex $\theta' \in \mathcal T$
with interior point
$\frac{ v_{i_2} + \cdots + v_{i_k} }{k-1}$.
On the other hand, 
$\frac{ v_{i_2} + \cdots + v_{i_k} }{k-1} = \frac{ (c_1-1) v_{j_1} + \cdots + c_l v_{j_l} }{k-1}$
lies in $\theta$, 
hence $\theta'$ is a face of $\theta$.
Since $\theta$ has $v_{i_1} = v_{j_1}$ as a vertex, 
it contradicts the assumption that $P$ is primitive.
\end{proof}

The K\"{a}hler cone $\mathcal K$ of $\widetilde{X^\vee}$
is a strongly convex maximal-dimensional cone in the \emph{secondary fan} 
$\mathcal{SF}(\triangle)$ of the polytope $\triangle$.
The secondary fan $\mathcal{SF}(\triangle)$ lies in 
$\HH^2(\widetilde{X^\vee}, \mathbb Z) \otimes_{\mathbb Z} \mathbb R 
\cong \mathbb R^{p+1} / \mathbb R^{n+1}$.
Its maximal-dimensional cones
correspond to the projective triangulations of $\triangle$. 
More precisely, 
given a weight vector $\omega = (\omega_0, \omega_1, \ldots, \omega_p) \in \mathbb R^{p+1}$, 
we may lift the integral point $v_i \in \mathcal A$ to the height $\omega_i$
and consider the convex hull $P_\omega$ of $(\omega_0, v_0), (\omega_1, v_1), \ldots, (\omega_p, v_p)$ 
in $\mathbb R^{n+1}$. 
For generic $\omega$, 
the lower envelope of $P_\omega$ induces a triangulation of $\triangle$ with vertices in $\mathcal A$.
Such a triangulation is called \emph{projective}.
The set of weight vectors inducing a fixed projective triangulation
are the interior points of a cone.
The cone contains the linear subspace $\mathbb R^{n+1}$, 
and its quotient in $\mathbb R^{p+1}/\mathbb R^{n+1}$ 
defines a cone in the secondary fan 
$\mathcal{SF} (\triangle)$.
It is under this correspondence that
the regular projective triangulation $\mathcal T$ of $\triangle$ yields the K\"{a}hler cone 
$\mathcal K_\mathcal T$ of 
$\widetilde{X^\vee_{\mathcal T}} = \mathbb P_{\Sigma_{\triangle, \mathcal T}}$. 

The secondary fan admits, as a refinement, the Gr\"{o}bner fan of the toric ideal
\begin{eqnarray}
\mathcal J_{\mathcal A} & = & \langle \prod_{i: \ell_i >0} y_i^{\ell_i} - \prod_{j: \ell_j <0} y_j^{-\ell_j} : \ell \in L \rangle 
\nonumber
\end{eqnarray}
in $\mathbb C[y_0, y_1, \ldots, y_p]$.
A generic weight vector $\omega \in \mathbb R^{p+1}$ 
defines a term order for $\mathcal J_{\mathcal A}$ 
such that the ideal generated by the leading terms 
$\text{LT}_\omega(\mathcal J_{\mathcal A}) := \langle \text{LT}_{\omega}(f) | f \in \mathcal J_{\mathcal A} \rangle$
is a monomial ideal. 
Here, the monomials $y_0^{\alpha_0} y_1^{\alpha_1} \cdots y_p^{\alpha_p}$ 
are ordered by their $\omega$-weights 
$\omega_0 \alpha_0 + \omega_1 \alpha_1 + \cdots + \omega_p \alpha_p$.
The equivalence class of weight vectors defining the same initial ideal 
yields a cone in the Gr\"{o}bner fan of $\mathcal J_{\mathcal A}$.
The K\"{a}hler cone
$\mathcal K_{\mathcal T}$ of $\mathbb P_{\Sigma_{\triangle, \mathcal T}}$ 
remains a cone in the Gr\"{o}bner fan of $\mathcal J_{\mathcal A}$. 
Indeed, 
a vector $\omega \in \mathcal K_{\mathcal T}^\circ$ 
defines a term order for $\mathcal J_{\mathcal A}$ 
with initial ideal $\text{LT}_{\omega} (\mathcal J_{\mathcal A})$ 
equal to the Stanley-Reisner ideal $\text{SR}_{\mathcal T}$ 
of the triangulation $\mathcal T$.
Recall that the Stanley-Reisner ideal of a triangulation $\mathcal T'$ 
is the ideal in $\mathbb C[y_0, y_1, \ldots, y_p]$ 
generated by the monomials $y_{i_1} y_{i_2} \cdots y_{i_m}$ 
for all $\{i_1, i_2, \ldots, i_m \} \notin \mathcal T'$.
Hence for $\omega \in \mathcal K_{\mathcal T}^\circ$, we have 
\begin{eqnarray}
\text{LT}_{\omega} (\mathcal J_{\mathcal A}) = \text{SR}_{\mathcal T} = 
\langle y_{i_1} \cdots y_{i_k} \,\, | \,\, P = \{ v_{i_1}, \ldots, v_{i_k} \} \,\, \text{a primitive collection of } \mathcal T \, \rangle.
\nonumber
\end{eqnarray}
In fact, the binomials 
\begin{eqnarray}
y_{i_1} \cdots y_{i_k} - y_{j_1}^{c_1} \cdots y_{j_l}^{c_l}, \quad \quad  \{v_{i_1}, \ldots, v_{i_k} \}  \text{  a primitive collection of $\mathcal T$}, 
\label{Grobner basis}
\end{eqnarray}
defined by the primitive relations in Lemma \ref{primitive}
form a Gr\"{o}bner basis of $\mathcal J_{\mathcal A}$ 
with respect to the term order $\omega \in \mathcal K_{\mathcal T}^\circ$.
In general,
if a weight vector $\omega$ defines a term order for the toric ideal $\mathcal J_{\mathcal A}$, 
it also induces a projective triangulation $\mathcal T_{\omega}$ of $\triangle$.
Furthermore,
the Stanley-Reisner ideal of $\mathcal T_\omega$ 
is equal to the radical of the initial ideal 
$\text{LT}_\omega(\mathcal J_{\mathcal A})$.
(See [St] for more details.)

The toric variety associated to the secondary fan $\mathcal{SF}(\triangle)$ is a compactification of
the $(p-n)$-dimensional torus 
$\HH^2(\widetilde{X^\vee}, \mathbb Z) \otimes_{\mathbb Z} \mathbb C^* 
\cong (\mathbb C^*)^{p+1}/ (\mathbb C^*)^{n+1}$.
The K\"{a}hler cone $\mathcal K_{\mathcal T}$, 
not necessarily regular or even simplicial, 
defines in it an affine toric subvariety. 
If $\mathcal K_{\mathcal T}$ is not regular, 
then we subdivide it into regular cones.
This yields a toric variety birational to the old one.
Any cone $\tau$ in the subdivision of $\mathcal K_{\mathcal T}$ now 
defines a smooth affine toric subvariety $U_\tau$.
Next, we construct a system of PDEs on $U_\tau$
which is equivalent to the GKZ-system
and discuss its solutions near the unique fixed point $x_\tau$ of $U_\tau$.

The dual cone $\tau^\vee$, containing the Mori cone $\mathcal M = \mathcal K^\vee$, 
is generated by a $\mathbb Z$-basis 
$l^{(1)}, \ldots, l^{(p-n)}$ of the lattice $L$.
The corresponding monomials $x_j = \prod_{i=0}^p a_i^{l^{(j)}_i}$, $1 \leq j \leq p-n$, 
give a set of smooth coordinates on $U_\tau$.
If $s(a)$ satisfies (2) and (3), 
then $a_0 s(a)$, invariant with respect to the $(\mathbb C^*)^{n+1}$-action (2) and (3) stipulate, 
is really a function of $x_1, \ldots, x_{p-n}$.
As for the differential operators in (1), 
we identify the toric ideal 
$\mathcal J_{\mathcal A} = \langle y^{\ell_+} - y^{\ell_-}: \ell \in L \rangle$ 
in $\mathbb C[y_0, y_1, \ldots, y_p]$
with the ideal generated by 
$\square_\ell = \prod_{i: \ell_i >0} \left( \frac{\partial}{\partial a_i} \right)^{\ell_i} - \prod_{j: \ell_j <0} \left( \frac{\partial}{\partial a_j} \right)^{-\ell_j}$
in $\mathbb C[\frac{\partial}{\partial a_0}, \frac{\partial}{\partial a_1}, \ldots, \frac{\partial}{\partial a_p}]$
for all $\ell \in L$.
Since (\ref{Grobner basis}) is a Gr\"{o}bner basis of $\mathcal J_{\mathcal A}$, 
in particular they generate $\mathcal J_{\mathcal A}$ as an ideal, 
so it is sufficient to consider the $\square_\ell$ for primitive relations. 
From here, 
it is not difficult to write down the corresponding differential operators in $x_1, \ldots, x_{p-n}$
with polynomial coefficients that annihilate $a_0 s(a)$. 
See [HLY1, 3.3] for more details. 

The point $x_\tau \in U_\tau$ where $x_1 = \cdots = x_{p-n} = 0$ 
is the large complex structure limit.
It was discovered that the solutions of the GKZ-system near $x_\tau$ 
are always supported on the Mori cone $\mathcal M$
regardless of which $\tau$ we pick in the subdivision of $\mathcal K_{\mathcal T}$.
In fact, there is an explicit formula. 
In order to describe the formula, 
we consider the cohomology ring of the smooth projective toric variety 
$\widetilde{X^\vee_{\mathcal T}} = \mathbb P_{\Sigma_{\triangle, \mathcal T}}$.
Let $D_1, \ldots, D_p$ be the $T$-invariant divisors of 
$\mathbb P_{\Sigma_{\triangle, \mathcal T}}$ 
corresponding to the integral points 
$v_1, \ldots, v_p$ of $\triangle$, 
and let $D_0$ be the artificial divisor. 
Then the cohomology ring of $\mathbb P_{\Sigma_{\triangle, \mathcal T}}$
is equal to 
\begin{eqnarray}
\HH^\bullet(\mathbb P_{\Sigma_{\triangle, \mathcal T}}, \mathbb Z) & = & 
\mathbb Z[ D_0, D_1, \ldots, D_p] / \mathcal I
\nonumber
\end{eqnarray}
where the ideal $\mathcal I$ is generated by

(a) the linear equivalence relations $D_0 + D_1 + \cdots + D_p =0$ and  
$\sum_{i=1}^p v_{i, j} D_i = 0$ for $1 \leq j \leq n$, 

\noindent
and

(b) the Stanley-Reisner ideal 
$$\langle D_{i_1} D_{i_2} \cdots D_{i_k} \,\, | \,\, 
\{ i_1, \ldots, i_k \} \subset \{1, 2, \ldots, p \} \text{  and   } v_{i_1}, \ldots, v_{i_k} \text{ do not generate a cone in   } \Sigma_{\triangle, \mathcal T} \rangle.
$$
([CLS, Theorem 12.4.4]).
Let $M = \mathcal M \cap \mathbb Z^{p+1}$ be the set of integral points 
in the Mori cone $\mathcal M$. 
Recall that the Mori cone is generated by primitive relations, 
and each primitive relation 
$\ell = (\ell_0, \ell_1, \ldots, \ell_p)$ has $\ell_0 \leq 0$, 
therefore all $\ell \in M$ satisfy this condition. 

\begin{theorem} 
[{[HLY1]}]
\label{solutions of GKZ}
The solutions to the GKZ-system on $U_\tau$ near the LCSL $x_\tau$ 
can be expressed as a vector-valued function 
\begin{eqnarray}
B_{X, \mathcal T}(a) 
& = & \frac{1}{a_0} \left(
\sum_{\ell = (\ell_0, \ell_1, \ldots, \ell_p) \in M}
O_\ell \,\,\,\, 
\prod_{i=0}^p a_i^{\ell_i} \right) \,\,\,\, 
\text{exp} \, \left( \sum_{i=0}^p (\log a_i) D_i \right)
\nonumber
\end{eqnarray}
with values in the cohomology ring
$\HH^\bullet(\widetilde{X^\vee_{\mathcal T}}, \mathbb C)$
where 
\begin{eqnarray}
O_\ell & = & \frac{ \prod_{i=1, \ell_i <0}^p D_i (D_i-1) \cdots (D_i +\ell_i +1)}
{\prod_{i=1, \ell_i \geq 0}^p (D_i +1) (D_i+2) \cdots (D_i +\ell_i)} 
(D_0-1)(D_0-2)\cdots (D_0 + \ell_0).
\nonumber
\end{eqnarray}
This is understood as follows:
for any $\alpha$ in the dual space of 
$\HH^\bullet(\widetilde{X^\vee_{\mathcal T}}, \mathbb C)$, 
the natural pairing 
$\langle B_{X, \mathcal T}(a), \alpha \rangle$ is a solution of the GKZ-system, 
moreover all solutions of the GKZ-system arise this way.
The $(D_i + k)$-term on the denominator of $O_\ell$ 
is understood as the geometric expansion in $D_i$.
\end{theorem}

\begin{proof}
First, we prove that $B_{X, \mathcal T}(a)$ satisfies the GKZ-equations 
(\ref{D_l})(\ref{t-invariance})(\ref{Euler}).
Equations (\ref{t-invariance})(\ref{Euler}) are easy to verify 
using the linear equivalence relations 
that $D_i, 0 \leq i \leq p$, satisfy
because
$\ell \in M \subset L$. 
For (\ref{D_l}), note that for $1 \leq i \leq p$, one has 
\begin{eqnarray}
\frac{\partial}{\partial a_i} B_{X, \mathcal T}(a)
& = & 
\frac{1}{a_0}
\left( \sum_{\ell \in M} O_\ell \frac{\ell_i}{a_i} a^\ell \right) 
\text{exp} \left(\sum_{i=0}^p (\log a_i) D_i \right)
+ \frac{1}{a_0} 
\left( \sum_{\ell \in M} O_\ell a^\ell \right)
\text{exp}  \left( \sum_{i=0}^p (\log a_i) D_i \right)
\frac{D_i}{a_i}
\nonumber \\
& = & 
\frac{1}{a_0} 
\left( \sum_{\ell \in M} O_\ell \frac{D_i + \ell_i}{a_i} a^\ell \right)
\text{exp} \left( \sum_{i=0}^p (\log a_i) D_i \right)
\nonumber
\end{eqnarray}
where 
$a^\ell = \prod_{i=0}^p a_i^{\ell_i}$.
For convenience, 
we extend the definition of $O_\ell$ to $O_d$
for $d \in \mathbb Z^{p+1}$ with $d_0 \leq 0$, 
and also the definition of $a^\ell$ to $a^d$.
A case-by-case analysis, depending on whether $\ell_i$ is positive, negative, or zero, 
reveals that 
$O_\ell (D_i+ \ell_i) = O_{\ell -e_i}$, 
where $e_i$ is the $i$-th standard basis vector of $\mathbb Z^{p+1}$.
Since $\frac{1}{a_i} a^\ell = a^{\ell - e_i}$, 
we have
\begin{eqnarray}
\frac{\partial}{\partial a_i} B_{X, \mathcal T}(a)
& = &
\frac{1}{a_0} \left( \sum_{d \in M-e_i} O_d a^d \right)
\text{exp} \left( \sum_{i=0}^p (\log a_i) D_i \right).
\label{derivative}
\end{eqnarray}
A similar computation shows that 
(\ref{derivative}) also holds for $i =0$.
Now taking $\ell \in L$, 
we have 
\begin{eqnarray}
\square_\ell  B_{X, \mathcal T}(a) 
& = & 
\left( \Pi_{i: \ell_i >0} \left( \frac{\partial}{\partial a_i} \right)^{\ell_i} - \Pi_{j: \ell_j <0} \left( \frac{\partial}{\partial a_j} \right)^{-\ell_j} \right) B_{X, \mathcal T}(a) 
\nonumber \\
& = & 
\frac{1}{a_0} \left( \sum_{d \in M - \sum_{i: \ell_i>0} \ell_i e_i} O_d a^d - 
\sum_{d \in M - \sum_{j: \ell_j <0} (- \ell_j) e_j} O_d a^d \right)
\text{exp} \left( \sum_{i=0}^p (\log a_i) D_i \right).
\nonumber
\end{eqnarray}
Denote 
$\ell^+ = \sum_{i: \ell_i >0} \ell_i e_i$ and 
$\ell^- = \sum_{j: \ell_j <0} -\ell_j e_j$, 
then $\ell = \ell^+ - \ell^-$ and 
$M - \ell^- = M - \ell^+ + \ell$.
Comparing the ranges of the two summations, 
we have
\begin{eqnarray}
\square_\ell  B_{X, \mathcal T}(a) & = &
\frac{1}{a_0} \left( \sum_{d \in (M \backslash M + \ell) - \ell^+} O_d a^d 
- \sum_{d \in (M+\ell \backslash M) - \ell^+}  O_d a^d \right)
\text{exp} \left( \sum_{i=0}^p (\log a_i) D_i \right).
\nonumber
\end{eqnarray}
For any $d \in (M \backslash M+ \ell) - \ell^+$, 
let $\gamma = d + \ell^+ - \ell = d + \ell^-$; 
then $\gamma \in L \backslash M$.
Since $\gamma \notin M$, 
there exists $\omega$ in the interior of the K\"{a}hler cone $\mathcal K$ such that 
$\langle \gamma, \omega \rangle <0$.
Write $\gamma = \gamma^+ - \gamma^-$. 
Then, with respect to the term order defined by $\omega$, 
the binomial $y^{\gamma^+} - y^{\gamma^-}$ in the toric ideal $\mathcal J_{\mathcal A}$
has leading term $y^{\gamma^-}$.
Hence, 
there exists a primitive collection 
$\{ v_{i_1}, \ldots, v_{i_k} \}$ of $\mathcal T$ such that 
$y^{\gamma^-}$ is divisible by $y_{i_1} \cdots y_{i_k}$.
In particular 
$\{ i_1, \ldots, i_k \} \subset \{ i: \gamma_i <0, 1\leq i \leq p \}$.
Since $\ell^-$ has nonnegative coordinates and $d = \gamma - \ell^-$, 
it follows that 
$\{ i_1, \ldots, i_k \} \subset \{i: d_i <0, 1 \leq i \leq p \}$.
For each $i$ such that $d_i<0$ and $1 \leq i \leq p$, 
one has a corresponding $D_i$ in the numerator of $O_d$.
Hence, 
$O_d$ contains $D_{i_1} \cdots D_{i_k}$ in its numerator, 
but $D_{i_1} \cdots D_{i_k}$ is equal to $0$ in 
$\HH^\bullet(\widetilde{X^\vee_{\mathcal T}}, \mathbb Z)$.
For $d \in (M+\ell \backslash M) - \ell^+$, 
the same argument applies replacing $\gamma$ by $d + \ell^+$.

The above shows that if we expand $B_{X, \mathcal T}(a)$ 
with respect to a $\mathbb C$-basis of 
$\HH^\bullet(\widetilde{X^\vee_{\mathcal T}}, \mathbb C)$, 
then the coefficients are solutions of the GKZ-equations.
It is not hard to see that 
these coefficient functions are also linearly independent, 
so we obtain $(\dim \HH^\bullet(\widetilde{X^\vee_{\mathcal T}}, \mathbb C))$-dimension of 
linearly independent solutions.
Since $\widetilde{X^\vee_{\mathcal T}}$ is complete and smooth, 
we have $\dim \HH^\bullet(\widetilde{X^\vee_{\mathcal T}}, \mathbb C) 
= \chi(\widetilde{X^\vee_{\mathcal T}})
= | \Sigma_{\triangle, \mathcal T} (n) |$
([CLS, Theorem 12.3.9, Theorem 12.3.11]).
Since $\triangle$ is reflexive and the triangulation $\mathcal T$ is regular, 
the number of $n$-dimensional cones in $\Sigma_{\triangle, \mathcal T}$ is equal to the normalized volume of $\triangle$.
To complete the proof of the theorem, 
it suffices to show that the holonomic D-module 
associated to the GKZ-hypergeometric system 
has rank $(n!) \text{vol}(\triangle)$.
Since we assumed that $\mathcal T$ is a regular triangulation,  
the lattice polytope $\triangle$ is normal ([CLS, Definition 2.2.9]), 
hence $1 \times (\triangle \cap \mathbb Z^n)$ generates the semigroup
$C(\triangle) \cap \mathbb Z^{n+1}$
where $C(\triangle)$ is the cone in $\mathbb R^{n+1}$ supported by $1 \times \triangle$. 
It remains to use [A, Corollary 5.11].
\end{proof}

We can now state the \emph{hyperplane conjecture} due to Hosono, Lian, and Yau in 1996.

\begin{conjecture} 
[{[HLY1]}]
\label{hyperplane}
Let $[\widetilde{Y^\vee_{\mathcal T}}] = -D_0$ be the Calabi-Yau class in 
$\HH^2(\widetilde{X^\vee_{\mathcal T}}, \mathbb C)$.
Then $B_{X, \mathcal T} (a) \cdot (-D_0)$
gives the complete set of period integrals of Calabi-Yau hypersurfaces $\widetilde{Y}_f$ in $\widetilde{X}$ near the LCSL.
\end{conjecture}

\begin{remark}
For any linear function $\alpha$ on $\HH^\bullet(\widetilde{X^\vee_{\mathcal T}}, \mathbb C)$, 
the pairing 
$\langle B_{X, \mathcal T} (a) \cdot (-D_0),  \alpha \rangle$ is equal to 
$\langle B_{X, \mathcal T} (a), (-D_0)^\vee \alpha \rangle$, 
where $(-D_0)^\vee \alpha$ is the pullback of $\alpha$ along the map
\begin{eqnarray}
(-D_0) \cdot & : & 
\HH^\bullet(\widetilde{X^\vee_{\mathcal T}}, \mathbb C) \to \HH^\bullet(\widetilde{X^\vee_{\mathcal T}}, \mathbb C).
\label{cupping}
\end{eqnarray}
Hence, the coefficient functions in $B_{X, \mathcal T} (a) \cdot (-D_0)$ 
comprise a subspace of the coefficient functions in $B_{X, \mathcal T} (a)$.
They are obtained by pairing $B_{X, \mathcal T} (a)$ with linear functions $\alpha$ 
which vanish on the kernel of the cupping map (\ref{cupping}). 
Therefore, the dimension of the space of coefficient functions in 
$B_{X, \mathcal T} (a) \cdot (-D_0)$
is equal to the rank of the map (\ref{cupping}).
\end{remark}

The main result of this paper is the following:

\begin{theorem} \label{main}
Conjecture \ref{hyperplane} holds for $X = \mathbb P^n$.
\end{theorem}

\begin{proof}
It was recently proved in [BHLSY, Corollary 1.5] that 
the extended GKZ-system is complete when $X$ is $\mathbb P^n$, 
that is, the period sheaf of Calabi-Yau hypersurfaces $Y_f$ in $\mathbb P^n$ coincides with the solution sheaf of the extended GKZ-system.
When $X$ is $\mathbb P^n$, the rank of the period sheaf 
is equal to 
\begin{eqnarray}
\nu_n & = & \frac{n}{n+1}(n^n - (-1)^n), 
\nonumber
\end{eqnarray}
which is the dimension of the middle vanishing cohomology
\begin{eqnarray}
H^{n-1}_{\text{van}}(Y_f, \mathbb C) & = &
\text{ker} (i_!: H^{n-1}(Y_f, \mathbb C) \to H^{n+1}(\mathbb P^n, \mathbb C) )
\nonumber
\end{eqnarray}
(see [BHLSY, Proposition 6.3, Corollary 6.4]).
Using these results,
it suffices to show that 
\begin{enumerate}[label=(\roman*)] 
\item
$B_{\mathbb P^n, \mathcal T} (a) \cdot (-D_0)$ satisfies the extended GKZ-system, 

\item
The rank of the map (\ref{cupping}) is equal to $\nu_n$.
\end{enumerate}
The proofs of (i) and (ii) are the content of the next two sections.
\end{proof}

\section{$\mathfrak g$-invariance of $B_{X,\mathcal T}(a) \cdot [\widetilde{Y^\vee_{\mathcal T}}]$}

Let $\mathbb P_\Sigma$ be a complete smooth toric variety defined by a fan $\Sigma$ in the scalar extension $N_{\mathbb R}$ of a lattice $N \cong \mathbb Z^n$.
We fix an identification $N = \mathbb Z^n$,
and let $t_1, \ldots, t_n$ be the standard coordinates of the torus
$T = N \otimes_{\mathbb Z} \mathbb C^* = (\mathbb C^*)^n$.
The Lie algebra of the full automorphism group of $\mathbb P_\Sigma$
is determined by the Lie algebra of $T$ and the root system 
$R(\Sigma)$, 
defined as
\begin{eqnarray}
R(\Sigma) & = &
\{ v \in \mathbb Z^n: \exists u \in \Sigma(1) \text{  with }  \langle v, u \rangle =-1 
\text{ and } \langle v, u' \rangle \geq 0 \text{  for all } u' \in \Sigma(1), u' \neq u \}
\nonumber
\end{eqnarray}
([O, Proposition 3.13], [HLY1, (2.10)]).
Here, 
$\Sigma(1)$ is the set of one-dimensional cones in the fan $\Sigma$; 
abusing notation, we identify the one-dimensional cones 
with their unique primitive generators. 
Fix $v \in R(\Sigma)$, and 
let $u$ be the corresponding element in $\Sigma(1)$.
Write $v = (v^1, \ldots, v^n)$, 
and $u = (u^1, \ldots, u^n)$.
Define
$t^v = t_1^{v^1} \cdots t_n^{v^n}$ and 
$\delta_u = \sum_{i=1}^n u^i t_i \frac{\partial}{\partial t_i}$.
Then the Lie algebra of the automorphism group of $\mathbb P_\Sigma$
can be described in terms of vector fields 
([O, Proposition 3.13], [HLY1, (2.11)])
\begin{eqnarray}
\text{Lie}(\text{Aut} \mathbb P_\Sigma) & = &
\left( \bigoplus_{i=1}^n \mathbb C t_i \frac{\partial}{\partial t_i} \right)  \oplus
\left(\bigoplus_{v \in R(\Sigma)} \mathbb C t^v \delta_u \right).
\nonumber
\end{eqnarray}
Recall that $\mathcal T^\vee$ is a regular projective triangulation of $\triangle^\vee$, 
and let $\Sigma_{\triangle^\vee, \mathcal T^\vee}$ be the refinement of the fan
$\Sigma_{\triangle^\vee}$ induced by $\mathcal T^\vee$.
Applying the above discussion 
to the complete regular fan $\Sigma_{\triangle^\vee, \mathcal T^\vee}$
and the toric variety $\widetilde X = \mathbb P_{\Sigma_{\triangle^\vee, \mathcal T^\vee}}$, 
we see that the one-dimensional cones in 
$\Sigma_{\triangle^\vee, \mathcal T^\vee}(1)$ 
which could appear in the definition of the root system 
$R(\Sigma_{\triangle^\vee, \mathcal T^\vee})$ 
can only be one of those in $\Sigma_{\triangle^\vee}(1)$.
By the duality of $\triangle$ and $\triangle^\vee$, 
we have the following description of $R(\Sigma_{\triangle^\vee, \mathcal T^\vee})$.

\begin{lemma}
Let $u_1, \ldots, u_r$ be the vertices of $\triangle^\vee$, 
$F_1, \ldots, F_r$ be the corresponding facets of $\triangle$.
Then $R(\Sigma_{\triangle^\vee, \mathcal T^\vee}) = \cup_{i=1}^r (\mathring{F_i} \cap \mathbb Z^n)$
where $\mathring{F_i}$ is the relative interior of $F_i$.
For $v \in \mathring{F_i} \cap \mathbb Z^n$, 
the corresponding vector field is then $t^v \delta_{u_i}$.
\end{lemma}

\begin{proof}
This is clear.
\end{proof}

Note that 
$\triangle = \{ v \in \mathbb R^n: \langle v, u_i \rangle \geq-1 \,\, \forall 1 \leq i \leq r \}$.
If $v \in \mathring{F_i} \cap \mathbb Z^n$, 
then $\langle v, u_i \rangle =-1$,
$\langle v, u_j \rangle \geq 0$ for all $j \neq i$, $1\leq j \leq r$.
Let $v' \in \triangle \cap \mathbb Z^n$.
Then 
\begin{itemize}
\item
$v' + v \in \triangle$ $\Leftrightarrow$ 
$\langle v', u_i \rangle \geq 0$ $\Leftrightarrow$  $v' \notin F_i$, 

\item
$v' + v \notin \triangle$ $\Leftrightarrow$  $\langle v', u_i \rangle =-1$
$\Leftrightarrow$  $v' \in F_i$.
\end{itemize}
Based on this observation, 
the following differential operator is well-defined
\begin{eqnarray}
\pounds_v & = & \sum_{v' \in \triangle \cap \mathbb Z^n, v' \notin F_i} 
(\langle v', u_i \rangle + 1) a_{v'} \frac{\partial}{\partial a_{v' + v}}, 
\qquad v \in \mathring{F_i} \cap \mathbb Z^n.
\label{D_v}
\end{eqnarray}

\begin{lemma}
The period integrals $\Pi_\cdot(a)$ of the Calabi-Yau hypersurfaces $\widetilde Y_f$ in $\widetilde X$ 
satisfy $\pounds_v \Pi_\cdot(a) = 0$ for all $v \in R(\Sigma_{\triangle^\vee, \mathcal T^\vee})$.
\end{lemma}

\begin{proof}
Recall $\omega = \frac{dt_1}{t_1} \wedge \cdots \wedge \frac{dt_n}{t_n}$, 
$f(t) = \sum_{v' \in \triangle \cap \mathbb Z^n} a_{v'} t^{v'}$, 
and $\omega_f$ is the Poincar\'{e} residue of $\frac{\omega}{f}$.
Fix $\gamma \in H_{n-1}(\widetilde Y_f, \mathbb C)$, 
and let $\tau(\gamma) \in H_n(\widetilde X \backslash \widetilde Y_f, \mathbb C)$
be the tube over $\gamma$.
Then $\Pi_\gamma(a) = \int_\gamma \omega_f = \int_{\tau(\gamma)} \frac{\omega}{f}$.
For any $v \in \mathring{F_i} \cap \mathbb Z^n$
with corresponding vector field $t^v \delta_{u_i}$, 
we have 
\begin{eqnarray}
0 & = & \int_{\tau(\gamma)} \text{Lie}_{t^v \delta_{u_i}} \left(\frac{\omega}{f} \right)
\nonumber
\end{eqnarray}
where $\text{Lie}_{t^v \delta_{u_i}}$ means the Lie derivative.
Straightforward computation shows that
\begin{eqnarray}
\text{Lie}_{t^v \delta_{u_i}} \frac{d t_k}{t_k} 
\quad = \quad d( \langle e_k, u_i \rangle t^v)
\quad = \quad \langle e_k, u_i \rangle t^v \sum_{l=1}^n \langle v, e_l \rangle \frac{d t_l}{t_l}
\nonumber
\end{eqnarray}
where $e_k$ is the $k$-th standard basis vector of $\mathbb R^n$.
Hence,
\begin{eqnarray}
\text{Lie}_{t^v \delta_{u_i}} \omega 
\quad = \quad ( \sum_{k=1}^n \langle e_k, u_i \rangle t^v \langle v, e_k \rangle ) \omega
\quad = \quad  \langle v, u_i \rangle t^v \omega
\quad = \quad - t^v \omega.
\nonumber
\end{eqnarray}
Similarly, 
$\text{Lie}_{t^v \delta_{u_i}} f = 
\sum_{v' \in \triangle \cap \mathbb Z^n} a_{v'} \langle v', u_i \rangle t^{v+v'}$.
Hence,
\begin{eqnarray}
\text{Lie}_{t^v \delta_{u_i}} \frac{\omega}{f} & = & 
- \frac{ \sum_{v' \in \triangle \cap \mathbb Z^n} a_{v'} \langle v', u_i \rangle t^{v + v'} }
{f^2} \omega + \frac{- t^v}{f} \omega
\quad = \quad 
- \frac{\sum_{v' \in \triangle \cap \mathbb Z^n} a_{v'} ( \langle v', u_i \rangle +1) t^{v+v'} }{ f^2 } \omega 
\nonumber \\
& = & \left( \sum_{v' \in  \triangle \cap \mathbb Z^n} a_{v'}  ( \langle v', u_i \rangle +1) \frac{\partial}{\partial a_{v+v'}} \right) \frac{\omega}{f}.
\nonumber
\end{eqnarray}
The lemma follows.
\end{proof}

\begin{define}
The PDE system consisting of equations
(\ref{D_l})(\ref{t-invariance})(\ref{Euler})(\ref{D_v})
is called the \emph{extended GKZ-system}.
\end{define}

It turns out that $B_{X, \mathcal T} (a) \cdot (-D_0)$ 
is always annihilated by the differential operators in (\ref{D_v}).

\begin{theorem}
$\pounds_v ( B_{X, \mathcal T} (a) \cdot (-D_0) ) = 0$ for all $v \in R(\Sigma_{\triangle^\vee, \mathcal T^\vee})$.
\end{theorem}

\begin{proof}
Suppose $v \in \mathring{F_i} \cap \mathbb Z^n$. 
For any $v' \in \triangle \cap \mathbb Z^n$, $v' \notin F_i$, from (\ref{derivative}) we have
\begin{eqnarray}
a_{v'} \frac{\partial}{\partial a_{v' + v}} B_{X, \mathcal T} (a)  & = &
\frac{1}{a_0} \left( \sum_{d \in M-e_{v'+v}} a_{v'} \, O_d \, a^d \right) 
\text{exp} \left( \sum_{v' \in \triangle \cap \mathbb Z^n} (\log a_{v'}) D_{v'} \right)
\nonumber
\end{eqnarray}
Clearly $a_{v'} a^d = a^{d + e_{v'}}$.
\footnote{In Section 2, 
we used $v_0, v_1, \ldots, v_p$ to denote the lattice points in $\triangle$
which yields a one-to-one correspondence between 
$\{0, 1, \ldots, p\}$ and  $\triangle \cap \mathbb Z^n$.
In the proof of this theorem, 
we use the letters $v', v''$ to enumerate the points in $\triangle \cap \mathbb Z^n$, 
so $e_{v'}$ denotes the standard basis vector in $\mathbb Z^{p+1}$
with $1$ in the $v'$-coordinate and $0$ everywhere else.
Similarly, $D_{v'}$ and $a_{v'}$ denote the divisor and the $a$-variable corresponding to $v'$.
For any $\gamma \in \mathbb Z^{p+1}$, 
we use $\gamma_{v'}$ to denote the $v'$-coordinate of $\gamma$.}

{\bf (a)}
If $v' \neq 0$, 
a case-by-case analysis, 
depending on the value of $d_{v'}$, 
shows  that
$$O_d = (D_{v'} + d_{v'} +1) O_{d+e_{v'}}.
$$
Hence
\begin{eqnarray}
a_{v'} \frac{\partial}{\partial a_{v' + v}} B_{X, \mathcal T}(a)  & = &
\frac{1}{a_0} \left( \sum_{d \in M-e_{v'+v} +e_{v'}} (D_{v'} + d_{v'} ) O_d a^d \right)
\text{exp} \left( \sum_{v' \in \triangle \cap \mathbb Z^n} (\log a_{v'}) D_{v'} \right).
\label{sum}
\end{eqnarray}
From $(v'+v)-v'=v-0$, 
we have $\ell_{v'}:=e_v+e_{v'}-e_{v'+v}-e_0 \in L$.
Then $$M - e_{v'+v} + e_{v'} = (M + \ell_{v'}) -e_v + e_0.$$
Since $v \in \mathring{F_i}$ and $v' \notin F_i$, 
$v$ and $v'$ do not span a cone in $\Sigma_{\triangle, \mathcal T}$.
In fact, 
$\{v, v' \}$ is a primitive collection with primitive relation $\ell_{v'}$.
Let $M_- = \{ \ell \in M: \ell_0 \leq-1 \}$, 
$M_0 = M \backslash M_- = \{ \ell \in M: \ell_0 = 0 \}$.
Clearly $\ell_{v'} \in M_-$, $M + \ell_{v'} \subset M_-$.
For reasons that will become clear later, 
we want to replace the sum 
$\sum_{d \in (M+\ell_{v'})-e_v + e_0}$ in (\ref{sum}) by 
a seemingly larger sum 
$\sum_{d \in M_- -e_v + e_0}$.
Note that for $d \in M_- -e_v + e_0$, 
one has $d_0 \leq 0$, 
hence $O_d$ is well-defined
for the additional values of $d$ we allow in the larger sum.
The following assertion
\begin{eqnarray}
(D_{v'} + d_{v'}) O_d = 0 & & \text{ for all } d \in (M_- \backslash M+ \ell_{v'})-e_v+e_0
\label{vanishing}
\end{eqnarray}
immediately implies that the two sums are equal.

\emph{Proof of the assertion:}
let $\gamma = d+ e_v - e_0 - \ell_{v'} = d - e_{v'} + e_{v' + v}$, 
then $\gamma \in L \backslash M$ and
$d= \gamma + e_{v'} - e_{v'+v}$.
As argued in the proof of Theorem \ref{solutions of GKZ}, 
there exists a primitive collection $P$ such that the set 
$\{v'': \gamma_{v''} <0 \}$ contains $P$.
As $d_{v'} = \gamma_{v'} +1$, we have two cases
\begin{itemize}
\item
$\gamma_{v'} \neq-1$, in which case 
$$P \subset \{v'': \gamma_{v''} <0 \} = \{v'': (\gamma+e_{v'})_{v''} <0 \}
\subset \{v'': d_{v''} <0 \}, 
$$
and hence $O_d = 0$,

\item
$\gamma_{v'} =-1$, in which case 
$$P \subset \{v'': \gamma_{v''} <0 \} = \{v'': (\gamma+e_{v'})_{v''} <0 \} \cup \{ v' \}
\subset  \{v'': d_{v''} <0 \} \cup \{ v' \}, 
$$
and hence $D_{v'} O_d =0$
(note that $d_{v'}=0$).
\end{itemize}
This concludes the proof of the assertion.

It follows from (\ref{vanishing}) that 
\begin{eqnarray}
a_{v'} \frac{\partial}{\partial a_{v' + v}} B_{X, \mathcal T}(a)  & = &
\frac{1}{a_0} \left( \sum_{d \in M_- -e_v +e_0} (D_{v'} + d_{v'} ) O_d a^d \right)
\text{exp} \left( \sum_{v' \in \triangle \cap \mathbb Z^n} (\log a_{v'}) D_{v'} \right).
\label{nonzero v'}
\end{eqnarray}

{\bf (b)} 
If $v' = 0$, 
then 
\begin{eqnarray}
a_0 \frac{\partial}{\partial a_v} B_{X, \mathcal T}(a) & = &
\frac{1}{a_0} \left( \sum_{d \in M - e_v} O_d a^d a_0 \right) 
\text{exp} \left( \sum_{v' \in \triangle \cap \mathbb Z^n} (\log a_{v'}) D_{v'} \right).
\nonumber
\end{eqnarray}
With $a^d a_0 = a^{d+e_0}$, 
we also want to write $O_d$ in terms of $O_{d+e_0}$.
However, 
for $d \in M-e_v$,
$O_{d+e_0}$ is only well-defined when $d_0 \leq-1$, 
or equivalently when $d \in M_- - e_v$.
Hence, 
we split the sum $\sum_{d \in M-e_v}$ into 
$\sum_{d \in M_- -e_v}$ and  $\sum_{d \in M_0 - e_v}$.
For $d \in M_- - e_v$, 
one has 
$O_d = (D_0 + d_0) O_{d+e_0}$ and therefore the first sum is  
\begin{eqnarray}
\sum_{d \in M_- -e_v} O_d a^d a_0 & = & 
\sum_{d \in M_- - e_v} (D_0 + d_0) O_{d+e_0} a^{d+e_0} 
\quad = \quad
\sum_{d \in M_- -e_v +e_0} (D_0 + d_0-1) O_d a^d.
\nonumber
\end{eqnarray}
Now, we have 
$a_0 \frac{\partial}{\partial a_v} B_{X, \mathcal T}(a) = I + II$, 
where 
\begin{eqnarray}
I & = & \frac{1}{a_0} 
\left( \sum_{d \in M_- -e_v +e_0} (D_0 + d_0-1) O_d a^d \right)
\text{exp} \left( \sum_{v' \in \triangle \cap \mathbb Z^n} (\log a_{v'}) D_{v'} \right), 
\nonumber \\
II & = & 
\left( \sum_{d \in M_0 - e_v} O_d a^d \right) 
\text{exp} \left( \sum_{v' \in \triangle \cap \mathbb Z^n} (\log a_{v'}) D_{v'} \right).
\nonumber
\end{eqnarray}

Combining {\bf (a)} and {\bf (b)}, we have 
\begin{eqnarray}
& &  \left( \sum_{v' \in \triangle \cap \mathbb Z^n, v' \notin F_i}
(\langle v', u_i \rangle +1) a_{v'} \frac{\partial}{\partial a_{v'+v}} \right)
B_{X, \mathcal T}(a) 
\nonumber \\
& = & 
\frac{1}{a_0} 
\left( \sum_{d\in M_- -e_v + e_0} 
\left( \underbrace{ \left(
\sum_{v' \in \triangle \cap \mathbb Z^n, v' \neq 0}
( \langle v', u_i \rangle + 1) (D_{v'} + d_{v'}) \right)
+ (D_0 + d_0-1) 
}_{III} \right)
O_d a^d \right)
\nonumber \\
& & 
\cdot \text{exp} \left( \sum_{v' \in \triangle \cap \mathbb Z^n} (\log a_{v'}) D_{v'} \right)
+ II.
\nonumber
\end{eqnarray}
We claim $III = 0$.
Indeed, we have the following linear equivalence relations
\begin{eqnarray}
\sum_{v' \in \triangle \cap \mathbb Z^n, v'  \neq 0} \langle v', u_i \rangle D_{v'} 
& = & 
\left( \sum_{v' \in \triangle \cap \mathbb Z^n, v' \neq 0} D_{v'} \right) + D_0 
\quad = \quad 0, 
\nonumber
\end{eqnarray}
hence
\begin{eqnarray}
III & = & \left( \sum_{v' \in \triangle \cap \mathbb Z^n, v' \neq 0} \langle v', u_i \rangle d_{v'} \right) + 
\left( \sum_{v' \in \triangle \cap \mathbb Z^n} d_{v'} \right)-1.
\nonumber
\end{eqnarray}
From $d \in M_- -e_v+e_0$, $M_- \subset L$, 
we deduce $\gamma:= d+e_v -e_0 \in L$, 
which implies 
$$
\sum_{v' \in \triangle \cap \mathbb Z^n} \gamma_{v'}  =0 
\quad \quad \quad \quad \text{and} \quad \quad
\sum_{v' \in \triangle \cap \mathbb Z^n, v' \neq 0} \gamma_{v'} v' =0.
$$
It follows that 
\begin{eqnarray}
\sum_{v' \in \triangle \cap \mathbb Z^n, v' \neq 0} \gamma_{v'} \langle v', u_i \rangle & = & 0
\nonumber
\end{eqnarray}
and 
\begin{eqnarray}
\sum_{v' \in \triangle \cap \mathbb Z^n} d_{v'} & = & 
\left(\sum_{v' \in \triangle \cap \mathbb Z^n} \gamma_{v'} \right)-1+1 = 0.
\nonumber
\end{eqnarray}
Finally, we have
$III = - \langle v, u_i \rangle - 1 = -(-1)-1=0$.

To summarize, 
we now have
$\pounds_v B_{X, \mathcal T}(a) = II$.
It remains to show $II \cdot (-D_0) =0$.
We make the following assertion:
\begin{eqnarray}
\forall v' \in \triangle \cap \mathbb Z^n, \, v' \notin F_i, \, v' \neq 0, \, \text{and } d \in M_0-e_v, 
\text{ one has   }
(D_{v'}+d_{v'}) O_d = 0.
\label{vanishing2}
\end{eqnarray}
The proof of (\ref{vanishing2}) is analogous to the proof of (\ref{vanishing}).
Let $\gamma= d-e_{v'}+e_{v'+v}+e_0 = d+e_v-\ell_{v'}$; 
then $\gamma \in M_0-\ell_{v'}$.
By definition, the $0$-th coordinate of $\ell_{v'}$ is either $-1$ or $-2$
depending on whether $v'+v \neq 0$ or $v'+v=0$,
whence $\gamma_0 \geq 1$.
Since $\ell_0 \leq 0$ for all $\ell \in M$, 
we deduce $\gamma \in L \backslash M$.
It implies that there exists a primitive collection $P$ such that 
$P \subset \{v'' \in \triangle \cap \mathbb Z^n: \gamma_{v''} <0 \}$.
If $\gamma_{v'} \neq-1$, 
then $P \subset \{v'': \gamma_{v''} <0\} = \{v'': (\gamma+e_{v'})_{v''} <0 \} \subset 
\{v'': d_{v''} <0 \}$, 
whence $O_d =0$.
If $\gamma_{v'} =-1$, 
then $d_{v'}=0$ and
$P \subset \{v'': \gamma_{v''} <0 \} = \{v'': (\gamma+e_{v'})_{v''}<0\} \cup \{v' \}
\subset \{v'': d_{v''} <0 \} \cup \{ v' \}$, 
whence $D_{v'} O_d =0$.
From (\ref{vanishing2}), we have
\begin{eqnarray}
0 & = & \sum_{v' \in \triangle \cap \mathbb Z^n, v' \notin F_i, v' \neq 0} (\langle v', u_i \rangle +1)
\left(\sum_{d \in M_0 - e_v} (D_{v'} + d_{v'}) O_d a^d \right)
\text{exp} \left( \sum_{v' \in \triangle \cap \mathbb Z^n} (\log a_{v'}) D_{v'} \right)
\nonumber \\
& = & 
\left( \sum_{d \in M_0 -e_v} \left( \underbrace{\sum_{v' \in \triangle \cap \mathbb Z^n, v' \neq 0}
(\langle v', u_i \rangle +1) (D_{v'}+d_{v'})}_{IV} \right)
O_d a^d \right)
\text{exp} \left( \sum_{v' \in \triangle \cap \mathbb Z^n} (\log a_{v'}) D_{v'} \right).
\nonumber
\end{eqnarray}
As before, one can easily show
\begin{eqnarray}
IV & = & 
\sum_{v' \in \triangle \cap \mathbb Z^n, v' \neq 0} \langle v', u_i \rangle D_{v'} +
\sum_{v' \in \triangle \cap \mathbb Z^n, v' \neq 0} D_{v'} +
\sum_{v' \in \triangle \cap \mathbb Z^n, v' \neq 0} \langle v', u_i \rangle d_{v'} +
\sum_{v' \in \triangle \cap \mathbb Z^n, v' \neq 0}  d_{v'} 
\nonumber \\
& = & 0 + (-D_0) - \langle v, u_i \rangle +(-1) = -D_0.
\nonumber
\end{eqnarray}
Hence,
\begin{eqnarray}
0 & = & 
\left( \sum_{d \in M_0 -e_v} (-D_0)
O_d a^d \right)
\text{exp} \left( \sum_{v' \in \triangle \cap \mathbb Z^n} (\log a_{v'}) D_{v'} \right) 
\nonumber \\
& = & (-D_0) \cdot II
\nonumber
\end{eqnarray}
\end{proof}

The full automorphism group of $\mathbb P^n$ is $PGL(n+1, \mathbb C)$
with Lie algebra $sl(n+1, \mathbb C)$.
In this case, 
the differential operators $\pounds_v$ in (\ref{D_v}) 
correspond to the infinitesimal actions of the classical root vectors in 
$sl(n+1, \mathbb C)$.

\begin{corollary}
When $X = \mathbb P^n$, the coefficient functions of 
$B_{\mathbb P^n, \mathcal T} (a) \cdot (-D_0)$
are solutions of the extended GKZ-system.
\end{corollary}

This completes the first half of the proof of Theorem \ref{main}.

\section{The number of solutions in $B_{\mathbb P^n, \mathcal T}(a) \cdot (-D_0)$}

In this section, we describe a method to efficiently 
compute the rank of the cupping map
$$ (-D_0) \cdot: \HH^\bullet(\widetilde{X^\vee_{\mathcal T}}, \mathbb C) \to 
\HH^\bullet(\widetilde{X^\vee_{\mathcal T}}, \mathbb C)
$$
and apply it in the case $X = \mathbb P^n$.
If $\widetilde{Y^\vee_{\mathcal T}}$ is ample in $\widetilde{X^\vee_{\mathcal T}}$, 
then by the Lefschetz decomposition, 
the rank of cupping with $[\widetilde{Y^\vee_{\mathcal T}}]= -D_0$ 
is equal to the codimension of the primitive cohomology in 
$\HH^\bullet(\widetilde{X^\vee_{\mathcal T}}, \mathbb C)$.
However, 
$[\widetilde{Y^\vee_{\mathcal T}}] = (\pi^\vee)^* [Y^\vee]$, 
being the pullback of the ample class 
$[Y^\vee]$ on $X^\vee$ 
along the crepant resolution 
$\pi^\vee: \widetilde{X^\vee_{\mathcal T}} \to X^\vee$, 
is in general only semi-ample, 
and so the hard Lefschetz theorem fails.
Nevertheless, 
if we can relate the cohomology of $\widetilde{X^\vee_{\mathcal T}}$ 
to the cohomology of $X^\vee$ and its subvarieties, 
then we can apply the ampleness of $[Y^\vee]$ and the hard Lefschetz theorem to compute this rank on $X^\vee$.

\begin{lemma}
Let $f: X \to Y$ be a continuous map of topological spaces, with 
$D^+(X), D^+(Y)$ the bounded-below derived categories of $\mathbb Q$-sheaves on $X$ and $Y$, 
$\omega \in \HH^2(Y, \mathbb Q)$, and 
$F^\cdot \in D^+(X)$.
Then under the canonical isomorphism
$\mathbb H^\cdot(X, F^\cdot) \cong \mathbb H^\cdot(Y, Rf_* F^\cdot)$, 
cupping with $f^* \omega \in \HH^2(X, \mathbb Q)$ on the left-hand side 
is the same as cupping with $\omega$ on the right-hand side.
\end{lemma}

\begin{proof}
As $\HH^2(Y, \mathbb Q) =\Ext^2_{\mathbb Q}( \underline{\mathbb Q}_Y, \underline{\mathbb Q}_Y) = \HHom_{D^+(Y)} (\underline{\mathbb Q}_Y, \underline{\mathbb Q}_Y[2])$, 
we may interpret $\omega$ as an arrow in $D^+(Y)$, 
i.e. $\omega: \underline{\mathbb Q}_Y \to \underline{\mathbb Q}_Y[2]$.
Similarly, 
$f^* \omega$ is regarded as an arrow in $D^+(X)$. 
Note that $f^* \omega: \underline{\mathbb Q}_X \to \underline{\mathbb Q}_X[2]$
is precisely the image of $\omega:  \underline{\mathbb Q}_Y \to \underline{\mathbb Q}_Y[2]$
under the functor $f^*: D^+(Y) \to D^+(X)$.
Under the isomorphisms
$\mathbb H^i(X, F^\cdot) = \HHom_{D^+(X)} (\underline{\mathbb Q}_X, F^\cdot[i])$, 
the cupping map 
\begin{eqnarray}
\smile\!\! f^* \omega & : & \mathbb H^i(X, F^\cdot) \to \mathbb H^{i+2}(X, F^\cdot)
\nonumber
\end{eqnarray}
is the composition product
\begin{eqnarray}
\HHom_{D^+(X)} (\underline{\mathbb Q}_X, F^\cdot[i]) & \to &
\HHom_{D^+(X)} (\underline{\mathbb Q}_X, F^\cdot[i+2]), 
\nonumber \\
\alpha & \mapsto & \alpha[2] \circ f^* \omega.
\nonumber
\end{eqnarray}
The same is true for cupping with $\omega$ on $Y$.
The functor $f^*$ is left adjoint to $Rf_*$, 
i.e. one has a family of isomorphisms
\begin{eqnarray}
\HHom_{D^+(X)}( f^* G^\cdot, F^\cdot) & \cong &  \HHom_{D^+(Y)} (G^\cdot, Rf_* F^\cdot)
\nonumber
\end{eqnarray}
for all $F^\cdot \in D^+(X)$, $G^\cdot \in D^+(Y)$, 
and it is functorial in both $F^\cdot$ and $G^\cdot$.
The canonical isomorphism
$\mathbb H^i(X, F^\cdot) \cong \mathbb H^i(Y, Rf_* F^\cdot)$
can be recast as a special case of the above isomorphism:
\begin{eqnarray}
\HHom_{D^+(X)}(f^* \underline{\mathbb Q}_Y, F^\cdot[i]) & \cong &
\HHom_{D^+(Y)}( \underline{\mathbb Q}_Y, Rf_* F^\cdot[i]).
\nonumber
\end{eqnarray}
The lemma is now a consequence of the functoriality in the factor $\underline{\mathbb Q}_Y$.
\end{proof}

\begin{corollary} \label{cupping ds}
The rank of (\ref{cupping}) is equal to the rank of 
\begin{eqnarray}
\smile\!\! [Y^\vee] & : & \mathbb H^\cdot(X^\vee, R\pi^\vee_* \underline{\mathbb Q}_{\widetilde{X^\vee_{\mathcal T}}}) \to 
\mathbb H^\cdot(X^\vee, R\pi^\vee_* \underline{\mathbb Q}_{\widetilde{X^\vee_{\mathcal T}}}).
\label{cupping downstairs}
\end{eqnarray}
\end{corollary}

To understand $R \pi^\vee_* \underline{\mathbb Q}_{\widetilde{X^\vee_{\mathcal T}}}$, we use 
Beilinson-Berstein-Deligne-Gabber's decomposition theorem ([BBD]).
Let $Z$ be an irreducible complex algebraic variety, 
and $D^b_c(Z)$ be the \emph{constructible bounded derived category} of $Z$.
An object in $D^b_c(Z)$ is a complex of sheaves of $\mathbb Q$-vector spaces on $Z$ 
such that its cohomology sheaves are constructible and have finite-dimensional stalks.
Let $U \subset Z$ be a nonsingular Zariski open subset and 
let $L$ be a local system of finite rank on $U$.
The \emph{intersection complex} $\IC_Z(L)$ is a complex of sheaves on $Z$, 
which extends the complex $L[\dim Z]$ on $U$ and is determined, 
up to unique isomorphism in $D^b_c(Z)$, 
by some support and co-support conditions.
Up to a dimensional shift, 
the hypercohomology of $\IC_Z(L)$ computes the intersection cohomology of $Z$ 
with coefficients in the local system $L$: 
$$\mathbb H^i(Z, \IC_Z(L)) = \IHH^{\dim Z + i} (Z, L).
$$
When $L = \underline{\mathbb Q}_U$, 
one obtains the intersection complex $\IC_Z$,
which computes the intersection cohomology groups of $Z$ up to a shift:
$$\mathbb H^i(Z, \IC_Z) = \IHH^{\dim Z +i} (Z, \mathbb Q).
$$
If $Z$ is smooth or rationally smooth, 
then there is a simple description of the intersection complex
$\IC_Z = \underline{\mathbb Q}_Z[\dim Z]$.
In this case, 
the rational intersection cohomology coincides with the rational singular cohomology:
$\IHH^\bullet(Z, \mathbb Q) = \HH^\bullet(Z, \mathbb Q)$.
The intersection cohomology groups of $Z$ is not a ring, 
however cup product with an ordinary cohomology class is well-defined
$$\smile: \HH^\bullet(Z, \mathbb Q) \times \IHH^\bullet (Z, \mathbb Q) \to \IHH^\bullet(Z, \mathbb Q),
$$
and this makes the intersection cohomology a module over ordinary cohomology.
If $\omega \in \HH^2(Z, \mathbb Q)$ is the class of an ample divisor on a projective variety $Z$, 
then for all $i, 0\leq i \leq \dim Z$, 
the $i$-th iterated cup product
\begin{eqnarray}
\omega^i: \IHH^{\dim Z-i} (Z, \mathbb Q) \to \IHH^{\dim Z+ i} (Z, \mathbb Q)
\nonumber
\end{eqnarray}
is an isomorphism. 
This is known as the hard Lefschetz theorem for intersection cohomology
([BBD]).

\begin{theorem}[Decomposition Theorem]
\label{BBD}
Let $f: X \to Y$ be a proper map of complex algebraic varieties.
Then there exists a finite collection of triples $(Y_a, L_a, d_a)$ made of 
locally closed smooth irreducible algebraic subvarieties $Y_a \subset Y$, 
semisimple local systems $L_a$ on $Y_a$, 
and integers $d_a$, 
such that one has a direct sum decomposition in the constructible bounded derived category $D^b_c(Y)$
\begin{eqnarray}
Rf_* \IC_X & \cong & \bigoplus_{(Y_a, L_a, d_a)} \IC_{\bar Y_a}(L_a) [-d_a].
\label{sheaf version}
\end{eqnarray}
Taking cohomology on both sides of (\ref{sheaf version}) yields the cohomological decomposition theorem
\begin{eqnarray}
\IHH^i(X, \mathbb Q) & \cong & \bigoplus_{(Y_a, L_a, d_a)} \IHH^{i-\dim X +\dim Y_a - d_a} (Y_a, L_a).
\nonumber
\end{eqnarray}
\end{theorem}

The readers can refer to [dCM] for an excellent survey on the decomposition theorem and perverse sheaves.
For our purposes, we need a version of the decomposition theorem for a toric fibration.
We recall some basic facts in toric geometry.
Let $X$ and $Y$ be complex toric varieties corresponding to the lattices $N_X$ and $N_Y$ and to the fans $\Sigma_X$ and $\Sigma_Y$.
Let $T_X$ and $T_Y$ be the maximal dimensional torus in $X$ and $Y$ respectively.
A \emph{toric map} is a morphism $f: X \to Y$
that induces a morphism of algebraic groups $g: T_X \to T_Y$ such that 
$f$ is $T_X$-equivariant with respect to the $T_X$-action on $Y$ via $g$.
Such an $f$ is determined by a unique linear map
$f_{N_{\mathbb R}}: (N_X)_{\mathbb R} \to (N_Y)_{\mathbb R}$ 
inducing $f_N: N_X \to N_Y$ such that for every cone $\sigma$ in $\Sigma_X$, 
there is a cone $\tau$ in $\Sigma_Y$ with $f_{N_{\mathbb R}} (\sigma) \subset \tau$.
Let $f_*: \Sigma_X \to \Sigma_Y$ be the map such that
$f_*(\sigma)$ is the smallest cone in $\Sigma_Y$ that contains $f_{N_{\mathbb R}}(\sigma)$.
Let $|\Sigma_X| = \cup_{\sigma \in \Sigma_X} \sigma$ 
be the support of the fan $\Sigma_X$, 
and likewise $|\Sigma_Y| = \cup_{\tau \in \Sigma_Y} \tau$ be the support of $\Sigma_Y$.
A toric map $f$ is proper if and only if 
$f_{N_{\mathbb R}}^{-1} (|\Sigma_Y|) = | \Sigma_X|$.
A proper toric map $f: X \to Y$ is called a \emph{fibration} 
if $f$ is surjective and has connected fibers, 
or equivalently if $f_N$ is surjective.
See Section 2 of [dCMM] for more details.

\begin{theorem}[{[dCMM, Theorem 5.1]}]
\label{toric fibration}
If $X$ and $Y$ are complex toric varieties and $f: X \to Y$ is a toric fibration,
then we have a decomposition
\begin{eqnarray}
Rf_* \IC_X & \cong & \bigoplus_{\tau \in \Sigma_Y} \bigoplus_{b \in \mathbb Z}
\IC_{V(\tau)}^{\oplus s_{\tau, b}} [-b]
\label{BBD for toric fibration}
\end{eqnarray}
where 
$V(\tau)$ is the closure of the orbit $O(\tau)$ corresponding to the cone $\tau \in \Sigma_Y$.
\end{theorem}

The point is: (a) the subvarieties appearing in the decomposition 
(\ref{BBD for toric fibration}) are torus-invariant, and (b) the intersection complexes that appear have constant coefficients.
The subvarieties $V(\tau)$ which appear in the decomposition are called the \emph{supports} of the toric fibration. 
For every $\tau \in \Sigma_Y$, 
set $\delta_\tau: = \sum_{b \in \mathbb Z} s_{\tau, b}$.
Then $V(\tau)$ is a support if and only if $\delta_\tau >0$.
If both $X$ and $Y$ are simplicial, 
then there is an easy formula for $\delta_\tau$.

If $f: X \to Y$ is a toric fibration, 
then for every $\tau \in \Sigma_Y$, 
we put 
\begin{eqnarray}
d_{\ell}(X/\tau) & = & \# \{ \sigma \in \Sigma_X \,\, | \,\, f_*(\sigma) = \tau, \,\, \text{codim}(\sigma)-\text{codim}(\tau) = \ell \}.
\nonumber
\end{eqnarray}

\begin{theorem}[{[dCMM, Theorem 6.1 (i)]}]
\label{total multiplicity}
Let $f: X \to Y$ be a toric fibration, and both $X$ and $Y$ are simplicial toric varieties, 
then for every $\tau \in \Sigma_Y$, one has
\begin{eqnarray}
\delta_\tau & = & \sum_{\sigma \subseteq \tau} (-1)^{\dim(\tau) - \dim(\sigma)} d_0(X/ \sigma).
\label{delta_tau}
\end{eqnarray}
\end{theorem}

For the rest of this section, 
$\triangle$ is the reflexive polytope in $\mathbb R^n$ with vertices
$$v_1 = (n,-1, \ldots,-1), \,\, v_2 = (-1, n, \ldots,-1), \,\, \ldots, \,\, v_n = (-1, \ldots,-1, n), \,\, v_{n+1} = (-1,-1, \ldots,-1).
$$
The dual polytope $\triangle^\vee$ has vertices 
$$(1, 0, \ldots, 0),\hspace{.1in} (0, 1, \ldots, 0), \hspace{.1in} \ldots \hspace{.1in} (0, 0, \ldots, 1), \hspace{.1in} (-1,-1,\cdots,-1).$$
The corresponding toric varieties are 
$X = \mathbb P_{\triangle} = \mathbb P^n$ and 
$X^\vee = \mathbb P_{\triangle^\vee} = \mathbb P^n/ (\mathbb Z_{n+1}^{n-1})$.
Let $\mathcal T$ be a regular projective triangulation of $\triangle$, 
\footnote{We will show in the appendix that such a triangulation exists.}
and let 
\begin{eqnarray}
\pi^\vee: \quad \widetilde{X^\vee} = \mathbb P_{\Sigma_{\triangle, \mathcal T}} 
& \to & X^\vee = \mathbb P_{\Sigma_\triangle}
\nonumber
\end{eqnarray}
be the resulting projective crepant resolution. 
We apply Theorem \ref{toric fibration} to $\pi^\vee$, 
which is clearly a fibration.
Since $\widetilde{X^\vee}$ is smooth and $X^\vee$ is simplicial, 
we have
$\IC_{\widetilde{X^\vee}} = \underline{\mathbb Q}_{\widetilde{X^\vee}}[n]$ and 
$\IC_{V(\tau)} = \underline{\mathbb Q}_{V(\tau)}[\text{codim}(\tau)]$ 
for all $\tau \in \Sigma_\triangle$.
After a shift, 
the decomposition (\ref{BBD for toric fibration}) in this case becomes
\begin{eqnarray}
R\pi^\vee_* \underline{\mathbb Q}_{\widetilde{X^\vee}} & \cong & 
\bigoplus_{\tau \in \Sigma_\triangle} \bigoplus_{b \in \mathbb Z}
\underline{\mathbb Q}_{V(\tau)}^{\oplus s_{\tau, b}} [\text{codim}(\tau)-n-b].
\label{sheaf}
\end{eqnarray}
It follows that
\begin{eqnarray}
\mathbb H^\bullet(X^\vee, R\pi^\vee_* \underline{\mathbb Q}_{\widetilde{X^\vee}})
& \cong & \bigoplus_{\tau \in \Sigma_\triangle} \bigoplus_{b \in \mathbb Z}
\HH^{\bullet - \dim(\tau) -b} (V(\tau), \mathbb Q)^{\oplus s_{\tau, b}}
\label{cohomology}
\end{eqnarray}
and this decomposition is compatible with the cup product $\smile\!\! [Y^\vee]$.
Therefore, the rank of (\ref{cupping downstairs}) is equal to 
\begin{eqnarray}
\sum_{\tau \in \Sigma_\triangle} \delta_\tau \,\, \text{rank} 
(\smile\!\! [Y^\vee]: \HH^\bullet(V(\tau), \mathbb Q) \to \HH^\bullet(V(\tau), \mathbb Q)).
\label{rank3}
\end{eqnarray}

The cones in $\Sigma_\triangle$ are in one-to-one correspondence with the proper subsets of 
$\{v_1, \ldots, v_n, v_{n+1} \}$.
For each $0 \leq i \leq n$, 
there are a total ${n+1 \choose i}$-number of $i$-dimensional cones. 
If $\tau_i$ is an $i$-dimensional cone in $\Sigma_\triangle$, 
it is easy to deduce from [CLS, Theorem 12.4.1] that
\begin{eqnarray}
(\II) \! \HH^\bullet(V(\tau_i), \mathbb Q) & = & \mathbb Q[t]/(t^{n-i+1})
\label{icoc}
\end{eqnarray}
where $t$ has degree 2.
Since $[Y^\vee]$ is ample, 
from the hard Lefschetz theorem for intersection cohomology, 
we know
\begin{eqnarray}
\text{rank} (\smile\!\! [Y^\vee] : \HH^\bullet(V(\tau_i), \mathbb Q) \to \HH^\bullet(V(\tau_i), \mathbb Q) )
& = & n-i.
\nonumber
\end{eqnarray}
On the other hand, 
the number $d_0(X^\vee/ \tau_i)$ is equal to the normalized volume of the face 
which supports $\tau_i$, 
hence 
\begin{eqnarray}
d_0(X^\vee/\tau_i) & = & \left \{ \begin{array}{ll}
(n+1)^{i-1} \quad & 1\leq i \leq n \\
1 \quad & i =0.
\end{array} \right.
\nonumber
\end{eqnarray}
Using formula (\ref{delta_tau}) as both $\widetilde{X^\vee}$ and $X^\vee$ are simplicial, 
we get
\begin{eqnarray}
\delta_{\tau_i} & = & 
\left( \sum_{k=1}^i {i \choose k} (-1)^{i-k} (n+1)^{k-1} \right)  + (-1)^i
\nonumber \\
& = & \frac{ (n+1-1)^i - (-1)^i }{n+1}  + (-1)^i 
\quad = \quad 
\frac{n^i + n (-1)^i}{n+1}.
\nonumber
\end{eqnarray}
Plugging the above formulas into (\ref{rank3}), 
we obtain
\begin{eqnarray}
& & \text{rank of (\ref{cupping})} \quad = \quad 
\text{rank of (\ref{cupping downstairs})} \quad = \quad 
\sum_{i=0}^n {n+1 \choose i} \frac{n^i + n (-1)^i}{n+1} (n-i) 
\nonumber \\
& = & 
\frac{1}{n+1} \left( 
n \sum_{i=0}^n {n+1 \choose i} n^i + n^2 \sum_{i=0}^n {n+1 \choose i} (-1)^i 
- (n+1) \sum_{i=1}^n {n \choose i-1} n^i - (n+1)n \sum_{i=1}^n {n \choose i-1} (-1)^i \right)
\nonumber  \\
& = & 
\frac{1}{n+1} \left(
n [ (n+1)^{n+1} - n^{n+1}] + n^2 [0-(-1)^{n+1}] -(n+1)n [ (n+1)^n -n^n] + (n+1)n[0-(-1)^n] \right)
\nonumber \\
& = & 
\frac{1}{n+1} \left(
-n^{n+2} + (-1)^n n^2 + (n+1) n^{n+1} - (-1)^n n(n+1) \right)
\nonumber \\
& = & \frac{1}{n+1} \left( n^{n+1} - (-1)^n n \right)
\quad = \quad \frac{n}{n+1} (n^n - (-1)^n) 
\quad = \quad \nu_n.
\nonumber
\end{eqnarray}
This finishes the proof of Theorem \ref{main}.

\section{The numbers $a(i(n+1))$}

In this section, we still focus on the case $X = \mathbb P^n$.
Let $\mathcal T$ be a regular projective triangulation of $\triangle$, 
and $\widetilde{X^\vee_{\mathcal T}} = \mathbb P_{\Sigma_{\triangle, \mathcal T}}$
be the corresponding projective crepant resolution of $X^\vee$.

\begin{define}
\label{a(s)}
For an integer $s$, $0 \leq s \leq (n-1)(n+1)$, 
we denote by $a(s)$ the number of integer solutions to the equation 
\begin{eqnarray}
k_0 + k_1 + \cdots + k_n & = & s
\nonumber
\end{eqnarray}
where each $k_i$ is between $0$ and $n-1$.
\end{define}

It was observed in [BHLSY, (2.4)] that 
\begin{eqnarray}
\nu_n & = & \sum_{i=0}^{n-1} a(i(n+1)).
\label{partition 1}
\end{eqnarray}
On the other hand, since
$\nu_n$ is equal to the rank of (\ref{cupping}), we also have
\begin{eqnarray}
\nu_n & = & 
\dim_{\mathbb Q} \HH^\bullet(\widetilde{X^\vee_{\mathcal T}}, \mathbb Q)/ \, \text{ker} \,  [\widetilde{Y^\vee_{\mathcal T}}].
\nonumber
\end{eqnarray}
Note that the cohomology ring
$\HH^\bullet(\widetilde{X^\vee_{\mathcal T}}, \mathbb Q) 
= \bigoplus_{i=0}^n \HH^{2i} (\widetilde{X^\vee_{\mathcal T}}, \mathbb Q)$ 
has only even-degree components
and cupping with $[\widetilde{Y^\vee_{\mathcal T}}]$ 
increases the degree by $2$.
This implies an obvious grading on 
$\HH^\bullet(\widetilde{X^\vee_{\mathcal T}}, \mathbb Q)/ \, \text{ker} \,  [\widetilde{Y^\vee_{\mathcal T}}]$
in degrees $0, 2, \ldots, 2n-2$, 
and, 
as a consequence, 
the following partition
\begin{eqnarray}
\nu_n & = & \sum_{i=0}^{n-1} \dim_{\mathbb Q} 
(\HH^\bullet(\widetilde{X^\vee_{\mathcal T}}, \mathbb Q)/ \, \text{ker} \,  [\widetilde{Y^\vee_{\mathcal T}}])_{2i}.
\label{partition 2}
\end{eqnarray}
Naturally one may ask if the partition in (\ref{partition 2}) is independent of $\mathcal T$, 
and if it is, whether the partitions in (\ref{partition 1}) and (\ref{partition 2}) coincide.
The next theorem answers both questions affirmatively. 

\begin{theorem}
\label{5.2}
For each $i$, $0 \leq i \leq n-1$, one has
\begin{eqnarray}
a(i(n+1)) & = & 
\dim_{\mathbb Q} (\HH^\bullet(\widetilde{X^\vee_{\mathcal T}}, \mathbb Q)/ \text{ker} \, [\widetilde{Y^\vee_{\mathcal T}}])_{2i}.
\nonumber
\end{eqnarray}
In particular, 
the right-hand side is independent of $\mathcal T$.
\end{theorem}

\begin{remark}
In fact, $a(i(n+1))$ are certain Hodge numbers of 
a smooth Calabi-Yau hypersurface in $\mathbb P^n$
(see Remark \ref{6.27}).
Once we derive a general formula for the graded dimension of 
$\HH^\bullet(\widetilde{X^\vee_{\mathcal T}}, \mathbb Q)/ 
\text{ker} [\widetilde{Y^\vee_{\mathcal T}}]$ 
in Theorem \ref{general formula}, 
the above identity follows as a corollary. 
\end{remark}

\begin{proof}
The proof again uses the decomposition theorem, 
but this time we keep track of the cohomological grading.
Let $t$ be an indeterminate, 
and form the generating series
\begin{eqnarray}
G(t) & = & \sum_{i \geq 0} \dim (\HH^\bullet(\widetilde{X^\vee_{\mathcal T}}, \mathbb Q)/ \text{ker} [\widetilde{Y^\vee_{\mathcal T}}])_i \,\, t^i.
\nonumber
\end{eqnarray}
It follows from Corollary \ref{cupping ds} and the decomposition (\ref{cohomology}) that
\begin{eqnarray}
G(t) & = & \sum_{i} \dim (\mathbb H^\bullet(X^\vee, R \pi^\vee_* \underline{\mathbb Q}_{\widetilde{X^\vee_{\mathcal T}}})/ \text{ker} [Y^\vee])_i \,\, t^i
\nonumber \\
& = & \sum_{\tau \in \Sigma_\triangle} \sum_{b \in \mathbb Z}
s_{\tau, b} \sum_i \dim (\HH^{\bullet-\dim(\tau)-b}(V(\tau), \mathbb Q)/\text{ker} [Y^\vee])_i \,\, t^i.
\nonumber
\end{eqnarray}
We have seen in (\ref{icoc}) that, 
for each $\tau \in \Sigma_\triangle$, 
the orbit closure $V(\tau)$ has the same rational (intersection) cohomology as the projective space of equal dimension.
The hard Lefschetz theorem implies that 
cupping with the ample class $[Y^\vee]$ 
kills the top-graded component of $(\II) \! \HH^\bullet(V(\tau), \mathbb Q)$ 
and nothing more.
Hence
\begin{eqnarray}
\sum_i \dim (\HH^{\bullet-\dim(\tau)-b}(V(\tau), \mathbb Q)/\text{ker} [Y^\vee])_i \,\, t^i
& = & t^{\dim(\tau)+b} (1+t^2 + \cdots +t^{2(n-\dim(\tau)-1)})
\nonumber
\end{eqnarray}
when $\dim(\tau)\leq n-1$, and $0$ when $\dim(\tau)=n$.
It follows that
\begin{eqnarray}
G(t) & = & 
\sum_{\tau \in \Sigma_\triangle, \dim(\tau) \leq n-1} 
\left( \sum_{b \in \mathbb Z} s_{\tau, b} \,\, t^{\dim(\tau)+b} \right)
\sum_{j=0}^{n-\dim(\tau)-1} t^{2j}.
\label{form1}
\end{eqnarray}
To compute $\sum_{b \in \mathbb Z} s_{\tau, b} \,\, t^{\dim(\tau)+b}$, 
we take the stalks of (\ref{cohomology}) at any point $x_\tau$ in the orbit $O(\tau)$:
\begin{eqnarray}
(R\pi^\vee_* \underline{\mathbb Q}_{\widetilde{X^\vee_{\mathcal T}}})_{x_\tau}
& \cong & 
\bigoplus_{\sigma \subseteq \tau} \bigoplus_{b \in \mathbb Z} 
\mathbb Q^{\oplus s_{\sigma, b}} [-\dim(\sigma)-b].
\nonumber
\end{eqnarray}
By proper base change, 
the Poincar\'{e} polynomial of the left-hand side is
\begin{eqnarray}
\sum_i \dim ( R^i\pi^\vee_* \underline{\mathbb Q}_{\widetilde{X^\vee_{\mathcal T}}} )_{x_\tau}  
\,\, t^i & = &
\sum_i \dim \HH^i({\pi^\vee}^{-1} (x_\tau), \mathbb Q) \,\, t^i,
\nonumber
\end{eqnarray}
which is equal to the Poincar\'{e} polynomial of the right-hand side
$\sum_{\sigma \subseteq \tau} \left( \sum_{b\in \mathbb Z} s_{\sigma, b} \,\, t^{\dim(\sigma)+b} \right)$.
By the Mobius inversion formula, we have
\begin{eqnarray}
\sum_{b \in \mathbb Z} s_{\tau, b} \,\, t^{\dim(\tau)+b} & = &
\sum_{\sigma \subseteq \tau} (-1)^{\dim(\tau)-\dim(\sigma)} 
\sum_i \dim \HH^i({\pi^\vee}^{-1}(x_\sigma), \mathbb Q) \,\, t^i
\label{form2}
\end{eqnarray}
(see [dCMM, Section 6]).
It follows from [dCMM, Corollary 4.7] that
\begin{eqnarray}
\sum_i \dim \HH^i({\pi^\vee}^{-1}(x_\sigma), \mathbb Q) \,\, t^i 
& = & \sum_{\ell \geq 0} d_\ell(\widetilde{X^\vee_{\mathcal T}}/\sigma) 
(t^2-1)^\ell.
\nonumber
\end{eqnarray}
Suppose $\dim \sigma \geq 1$, and let 
$\theta$ be the face of $\triangle$ supporting $\sigma$.
The regular triangulation $\mathcal T$ 
induces a regular triangulation 
$\mathcal T_\theta$ of $\theta$.
Here, regular means each simplex in $\mathcal T_\theta$ has normalized volume $1$.
Then $d_\ell(\widetilde{X^\vee_{\mathcal T}}/ \sigma)$ 
is the number of $(\dim(\sigma)-\ell-1)$-dimensional simplices 
in $\mathcal T_\theta$ that are not contained in the boundary of $\theta$.
It turns out that these numbers are independent of the regular triangulation $\mathcal T$.
More precisely, let
\begin{eqnarray}
P_\theta(t) & := & \sum_{k \geq 0} l(k \theta) t^k
\nonumber
\end{eqnarray}
be the Ehrhart series of $\theta$
where $l(k \theta)$ denotes the number of lattice points in $k \theta$, 
then 
\begin{eqnarray}
(1-t)^{\dim \sigma} P_\theta(t) & = &
\sum_{\ell=0}^{\dim \sigma-1} d_\ell(\widetilde{X^\vee_{\mathcal T}}/ \sigma) \, \, (t-1)^\ell
\nonumber
\end{eqnarray}
(see the proof of [BD, Theorem 4.4], also Proposition \ref{Ehrhart}).
If $\dim \sigma=m \geq 1$, 
then $\theta = (n+1) \theta'$
where $\theta'$ is an $(m-1)$-dimensional regular simplex.
It follows that
\begin{eqnarray}
l(k \theta) & = &  l(k (n+1) \theta') 
\quad = \quad { k(n+1) + m-1 \choose m-1}.
\nonumber
\end{eqnarray}
Set 
\begin{eqnarray}
g_m(t) & := &
(1-t)^m \sum_{k \geq 0} { k(n+1) + m-1 \choose m-1} \, t^k, 
\qquad m \geq 1, 
\nonumber
\end{eqnarray}
then for any $\sigma \in \Sigma_\triangle$ with $\dim \sigma \geq 1$, 
we have
\begin{eqnarray}
\sum_i \dim \HH^i({\pi^\vee}^{-1}(x_\sigma), \mathbb Q) \,\, t^i 
& = & g_{\dim \sigma} (t^2).
\nonumber
\end{eqnarray}
This formula also holds for the unique $0$-dimensional cone
if we set $g_0(t) =1$.
From (\ref{form2}), we now have
\begin{eqnarray}
\sum_{b \in \mathbb Z} s_{\tau, b} \,\, t^{\dim(\tau)+b} & = &
\sum_{j=0}^{\dim \tau} (-1)^{\dim \tau-j} {\dim \tau \choose j} g_j(t^2).
\nonumber
\end{eqnarray}
Plugging the above into (\ref{form1}) yields
\begin{eqnarray}
G(t) & = & 
\sum_{i=0}^{n-1} {n+1 \choose i}
\left( \sum_{j=0}^i (-1)^{i-j} {i \choose j} g_j(t^2) \right)
\left( \sum_{k=0}^{n-i-1} t^{2k} \right).
\nonumber
\end{eqnarray}
Set 
\begin{eqnarray}
\tilde G(t) & :=& 
\sum_{i=0}^{n-1} {n+1 \choose i}
\left( \sum_{j=0}^i (-1)^{i-j} {i \choose j} g_j(t) \right)
\left( \sum_{k=0}^{n-i-1} t^{k} \right), 
\nonumber
\end{eqnarray}
then $G(t) = \tilde G(t^2)$.

For $0\leq j\leq n-1$, the coefficient of $g_j(t)$ in $\tilde G(t)$ is
\begin{eqnarray}
& & \sum_{i=j}^{n-1} {n+1 \choose i} (-1)^{i-j} {i \choose j}
\left( \sum_{k=0}^{n-i-1} t^{k} \right)
\nonumber \\
& = &  {n+1 \choose j} \sum_{i=j}^{n-1} (-1)^{i-j} {n+1-j \choose i-j} 
\left( \sum_{k=0}^{n-i-1} t^{k} \right) 
\nonumber \\
& = & {n+1 \choose j} \sum_{i=0}^{n-j-1} (-1)^i {n-j+1 \choose i} 
\left( \sum_{k=0}^{n-i-j-1} t^{k} \right) 
\nonumber \\
& = & {n+1 \choose j} \sum_{k=0}^{n-j-1} 
\left( \sum_{i=0}^{n-j-1-k} (-1)^i {n-j+1 \choose i} \right) \, t^k
\nonumber \\
& = & {n+1 \choose j} \underbrace{\sum_{k=0}^{n-j-1} 
(-1)^{n-j-1-k} {n-j \choose k+1} t^k}_I
\nonumber
\end{eqnarray}
where the first equality uses ${n+1 \choose i} {i \choose j} = {n+1 \choose j}{n+1-j \choose i-j}$,
and the last equality uses 
$\sum_{i=0}^b (-1)^i {a \choose i} = (-1)^b {a-1 \choose b}$
for $a > b \geq 0$.
Note that 
\begin{eqnarray}
I & = & \frac{\sum_{k=1}^{n-j} (-1)^{n-j-k} {n-j \choose k} t^k}{t}
\quad = \quad \frac{ (t-1)^{n-j} - (-1)^{n-j} }{t}
\nonumber \\
& = & (-1)^{n-j-1} \frac{ 1- (1-t)^{n-j} }{1-(1-t)} 
\quad = \quad (-1)^{n-j-1} \sum_{i=0}^{n-j-1} (1-t)^i,
\nonumber
\end{eqnarray}
hence, to summarize, the coefficient of $g_j(t)$ in $\tilde G(t)$ is 
$\displaystyle{(-1)^{n-j-1} {n+1 \choose j} \sum_{i=0}^{n-j-1} (1-t)^i}$.

By the definition of $g_j(t)$, we now have
\begin{eqnarray}
\tilde G(t) & = & (-1)^{n-1} \sum_{i=0}^{n-1}  (1-t)^i
\nonumber \\
& & \quad + \quad 
\sum_{j=1}^{n-1} (-1)^{n-j-1} {n+1 \choose j} \left( \sum_{i=0}^{n-j-1} (1-t)^i \right)
(1-t)^j \sum_{k \geq 0} {k(n+1) +j-1 \choose j-1} t^k
\nonumber \\
& = & (-1)^{n-1} \sum_{i=0}^{n-1}  (1-t)^i 
\nonumber \\
& & \quad + \quad
\sum_{j=1}^{n-1} (-1)^{n-j-1} {n+1 \choose j} \left( \sum_{i=j}^{n-1} (1-t)^i \right) 
\sum_{k \geq 0} {k(n+1) +j-1 \choose j-1} t^k.
\nonumber
\end{eqnarray}

The goal is to prove
$G(t) = \displaystyle{\sum_{i=0}^{n-1} a(i(n+1)) t^{2i}}$ or
$\tilde G(t) = \displaystyle{\sum_{i=0}^{n-1} a(i(n+1)) t^i}$.
It is equivalent to prove 
\begin{eqnarray}
\tilde G(t^{n+1}) & = & \sum_{i=0}^{n-1} a(i(n+1)) t^{i(n+1)}.
\nonumber
\end{eqnarray}
The right-hand side is obtained from the generating series
\begin{eqnarray}
\sum_{s=0}^{(n-1)(n+1)} a(s) t^s & = & (1+t+t^2+\cdots+t^{n-1})^{n+1}
\nonumber
\end{eqnarray}
by deleting those monomials $a(s) t^s$ for which $n+1 \nmid s$.
Since
\begin{eqnarray}
\tilde G(t^{n+1}) & = & 
(-1)^{n-1} \sum_{i=0}^{n-1}  (1-t^{n+1})^i 
\nonumber \\
& & \quad + \quad
\sum_{j=1}^{n-1} (-1)^{n-j-1} {n+1 \choose j} \left( \sum_{i=j}^{n-1} (1-t^{n+1})^i \right) 
\sum_{k \geq 0} {k(n+1) +j-1 \choose j-1} t^{k(n+1)}, 
\nonumber
\end{eqnarray}
we put
\begin{eqnarray}
F(t) & := & 
(-1)^{n-1} \sum_{i=0}^{n-1}  (1-t^{n+1})^i 
\nonumber \\
& & \quad + \quad
\sum_{j=1}^{n-1} (-1)^{n-j-1} {n+1 \choose j} \left( \sum_{i=j}^{n-1} (1-t^{n+1})^i \right) 
\underbrace{\sum_{k \geq 0} {k +j-1 \choose j-1} t^{k}}_{II}.
\nonumber
\end{eqnarray}
Then $\tilde G(t^{n+1})$ is obtained from $F(t)$ by throwing out monomials with exponents not divisible by $n+1$.
We conclude that it suffices to prove
\begin{eqnarray}
F(t) & = & (1+t+t^2+\cdots+t^{n-1})^{n+1}.
\nonumber
\end{eqnarray}

Since $II = \displaystyle{\frac{1}{(1-t)^j}}$, 
we have
\begin{eqnarray}
F(t) & = & 
\sum_{j=0}^{n-1} (-1)^{n-j-1} {n+1 \choose j} 
\left( \sum_{i=j}^{n-1} (1-t^{n+1})^i \right) \frac{1}{(1-t)^j}
\nonumber \\
& = & 
\sum_{j=0}^{n-1} (-1)^{n-j-1} {n+1 \choose j} 
\sum_{i=j}^{n-1} (1+t+\cdots+t^n)^i (1-t)^{i-j}
\nonumber \\
& = & \sum_{i=0}^{n-1} \, \square^i \,\, \sum_{j=0}^i (-1)^{n-j-1} {n+1 \choose j} (1-t)^{i-j}
\nonumber
\end{eqnarray}
where $\square = 1+t+\cdots+t^n$.
Using the identity
\begin{eqnarray}
\sum_{i=0}^m {M \choose i} (t-1)^{m-i} & = &
\sum_{j=0}^m {M-1-j \choose m-j} t^j
\quad \quad \text{for } 0 \leq m < M
\nonumber
\end{eqnarray}
which, for example, 
can be proved by induction on $M-m$, 
we get
\begin{eqnarray}
F(t) & = & 
\sum_{i=0}^{n-1} (-1)^{n-i-1} \square^i \sum_{j=0}^i {n-j \choose i-j} t^j
\quad = \quad 
\sum_{j=0}^{n-1} \left( \sum_{i=j}^{n-1} (-1)^{n-i-1} \square^i {n-j \choose i-j} \right) t^j
\nonumber \\
& = &
\sum_{j=0}^{n-1} \square^j 
\left( \sum_{i=0}^{n-j-1} (-1)^{n-i-j-1} \square^i {n-j \choose i} \right) t^j
\nonumber \\
& = & 
\sum_{j=0}^{n-1} t^j \square^j
[ \square^{n-j} - (\square-1)^{n-j} ]
\nonumber \\
& = & 
\sum_{j=0}^{n-1} ( t^j \square^n - t^n \square^j \triangle^{n-j} )
\qquad \qquad \qquad \text{where   } \triangle = 1+t+\cdots+t^{n-1}
\nonumber \\
& = & 
\triangle \square^n - (\square-\triangle) \sum_{j=0}^{n-1} \square^j \triangle^{n-j}
\nonumber \\
& = & 
\triangle \square^n - (\triangle \square^n - \triangle^{n+1}) 
\quad = \quad \triangle^{n+1}.
\nonumber
\end{eqnarray}
We are done.
\end{proof}

\section{A general formula}

In this section, we prove a general formula for the rank of the cupping map (\ref{cupping}).
This is one of the crucial steps towards proving the general hyperplane conjecture.

Recall that $\triangle, \triangle^\vee$ are a pair of $n$-dimensional reflexive polytopes, and 
$X=\mathbb P_\triangle$, $X^\vee = \mathbb P_{\triangle^\vee}$ are the associated projective toric varieties.
Let $Y$ be a $\triangle$-regular Calabi-Yau hypersurface in $X$, i.e. it intersects each toric orbit transversely, and similarly 
let $Y^\vee$ be a $\triangle^\vee$-regular Calabi-Yau hypersurface in $X^\vee$.
Suppose $\mathcal T$, resp. $\mathcal T^\vee$,
is a regular projective triangulation of $\triangle$, resp. $\triangle^\vee$.
Denote by $\Sigma_{\triangle, \mathcal T}$ 
and $\Sigma_{\triangle^\vee, \mathcal T^\vee}$
the refinements of the fans $\Sigma_\triangle$ and $\Sigma_{\triangle^\vee}$
that $\mathcal T$ and $\mathcal T^\vee$ induce.
Let $\pi: \widetilde X \to X$ and 
$\pi^\vee: \widetilde{X^\vee} \to X^\vee$
be the corresponding projective crepant resolutions
where $\widetilde X = \mathbb P_{\Sigma_{\triangle^\vee, \mathcal T^\vee}}$ and 
$\widetilde{X^\vee} = \mathbb P_{\Sigma_{\triangle, \mathcal T}}$.
Also, 
let $\widetilde Y$ and $\widetilde {Y^\vee}$ be the inverse images of 
$Y$ and $Y^\vee$ under $\pi$ and $\pi^\vee$.
Note that both $\widetilde Y$ and $\widetilde{Y^\vee}$ are smooth Calabi-Yau hypersurfaces
and they intersect the orbits of $\widetilde X$ and $\widetilde{X^\vee}$ 
in smooth subvarieties of codim 1 or the empty set.

As seen in Section 4, 
the number of linearly independent period integrals for the $\widetilde Y$-family
predicted by the hyperplane conjecture is equal to the rank of 
$\smile\!\! [\widetilde{Y^\vee}]: \HH^\bullet(\widetilde{X^\vee}, \mathbb Q) \to \HH^\bullet(\widetilde{X^\vee}, \mathbb Q)$, 
or equivalently the rank of 
$\smile\!\! [Y^\vee]: \mathbb H^\bullet( X^\vee, R \pi^\vee_* \underline{\mathbb Q}_{\widetilde{X^\vee}})
\to \mathbb H^\bullet( X^\vee, R \pi^\vee_* \underline{\mathbb Q}_{\widetilde{X^\vee}})$. 

\begin{notation}
For our purposes, 
it is more convenient to consider the shifted IC-complex
$\mathcal I_Z := \IC_Z[-\dim Z]$
of a complex algebraic variety $Z$.
Then $\IHH^\bullet(Z, \mathbb Q) = \mathbb H^\bullet(Z, \mathcal I_Z)$.
\end{notation}

Applying the decomposition theorem \ref{toric fibration} to the map 
$\pi^\vee: \widetilde{X^\vee} \to X^\vee$ 
and applying the cohomological shifts, 
we obtain
\begin{eqnarray}
R\pi^\vee_* \underline{\mathbb Q}_{\widetilde{X^\vee}} & \cong &
\bigoplus_{\sigma \in \Sigma_\triangle} \bigoplus_{b \in \mathbb Z}
\mathcal I_{V(\sigma)}^{ \oplus s_{\sigma, b}} [ -\dim \sigma -b].
\nonumber
\end{eqnarray}

\begin{notation}
\label{multiplicity series}
Define $S_\sigma(t): = \sum_{b \in \mathbb Z} s_{\sigma, b} t^{\dim \sigma + b}$
for any $\sigma \in \Sigma_\triangle$.
Note that $S_\sigma(t)$ encodes the multiplicity with which $\mathcal I_{V(\sigma)}$ 
appears in the decomposition of 
$R\pi^\vee_* \underline{\mathbb Q}_{\widetilde{X^\vee}}$
according to the cohomological grading.
It differs from the $S$-polynomial in [dCMM, Section 7] by a degree shift.
\end{notation}

\begin{notation}
Define $P_{\pi^\vee, \tau}(t): = 
\sum_{k \in \mathbb Z} \mathcal H^k( R\pi^\vee_* \underline{\mathbb Q}_{\widetilde{X^\vee}} )_{x_\tau} t^k$
where $\tau \in \Sigma_\triangle$ and 
$x_\tau$ is an arbitrary point in the orbit $O(\tau)$ of $X^\vee$.
By proper base change, one has 
$P_{\pi^\vee, \tau}(t) = \sum_{k \in \mathbb Z}
\HH^k( {\pi^\vee}^{-1} (x_\tau), \mathbb Q ) t^k$.
This also differs from the $P$-polynomial in [dCMM, Section 7] by a degree shift.
\end{notation}

\begin{notation}
For a pair of cones $\sigma \subseteq \tau$ in $\Sigma_\triangle$, 
define $R_{\sigma, \tau} (t): = \sum_{k \in \mathbb Z} 
\mathcal H^k( \mathcal I_{V(\sigma)} )_{x_\tau} t^k$.
This is known as the local intersection cohomology Poincar\'{e} polynomial of $V(\sigma)$
(see [F]).
Again, our $R_{\sigma, \tau}(t)$ differs from the $R$-polynomial in 
[dCMM, Section 7] by a shift in degree.
\end{notation}

With the shifted notations, we now have
\begin{eqnarray}
P_{\pi^\vee, \tau}(t) & = & 
\sum_{\sigma \subseteq \tau, \sigma \in \Sigma_\triangle} S_\sigma(t) R_{\sigma, \tau}(t)
\nonumber
\end{eqnarray}
for any $\tau \in \Sigma_\triangle$
(see the proof of [dCMM, Theorem 7.3]).

All these polynomials admit combinatorial descriptions.
We recall some basic facts about partially ordered sets (poset), 
and certain polynomials associated to a finite Eulerian poset
(see [S1], [S3], [BBo2, Section 2], [dCMM, Section 6]).

Let $(\mathcal P, \leq)$ be a finite poset, and let 
$\text{Int}(\mathcal P) = \{ [a, b]: a \leq b, a, b \in \mathcal P \}$ 
denote the set of intervals of $\mathcal P$.
Let $K$ be a commutative ring with identity.
One defines the \emph{incidence algebra}
$\mathbb I(\mathcal P, K)$ consisting of the set of functions
$f: \text{Int}(\mathcal P) \to K$, 
$[a, b] \mapsto f(a, b)$, 
equipped with the \emph{convolution} product
\begin{eqnarray}
f \ast g (a, b) & := &
\sum_{a \leq c \leq b} f(a, c) g(c, b).
\nonumber
\end{eqnarray}
This gives $\mathbb I(\mathcal P, K)$ the structure of an associative $K$-algebra
with an identity element $\delta$, 
defined by $\delta(a, b)=1$ if $a=b$, 
and $\delta(a, b)=0$ otherwise.
It is easy to prove that $f$ is invertible in $\mathbb I(\mathcal P, K)$ if and only if 
$f(a, a)$ is a unit of $K$ for every $a \in \mathcal P$.
Furthermore, if $f$ is invertible with inverse $g$ and 
$\phi, \psi: \mathcal P \to K$ are functions such that 
$\phi(a) = \sum_{b \leq a} \psi(b) f(b, a)$, 
then $\psi(b) = \sum_{a \leq b} \phi(a) g(a, b)$.
For our purposes, $K$ could be $\mathbb Z$ or $\mathbb Z[t]$.

From now on, we also assume that the finite poset $(\mathcal P, \leq)$ 
has a unique minimal element $\hat 0$,
a unique maximal element $\hat 1$, 
and that every maximal chain in $\mathcal P$ has the same length $d$.
We call $d$ the \emph{rank} of $\mathcal P$ and denote it by 
$\text{rk} (\mathcal P)$.
All the posets we consider in this section satisfy this condition.
If $\mathcal P$ is such a finite poset, 
then every closed interval $[a, b]$ of $\mathcal P$ also satisfies this condition.
Define the rank function as follows
$\rho: \mathcal P \to \{0, 1, 2, \ldots, d \}$; 
$a \mapsto \text{rk} [\hat 0, a]$.
Then $\text{rk} [a, b] = \rho(b) - \rho(a)$ for all $a \leq b$.

Let $(\mathcal P, \leq)$ be a finite poset as above, 
and let $K = \mathbb Z$.
The zeta function $\zeta: \text{Int} (\mathcal P) \to \mathbb Z$
given by $\zeta(a, b) =1$ for all $[a, b] \in \text{Int} (\mathcal P)$
is obviously invertible with respect to the convolution product
of $\mathbb I(\mathcal P, \mathbb Z)$.
The inverse of $\zeta$, denoted by 
$\mu_{\mathcal P}$, is called the 
\emph{Mobius function} of $\mathcal P$.
The poset $\mathcal P$ is called \emph{Eulerian} if 
$\mu_{\mathcal P} (a, b) = (-1)^{\rho(b)-\rho(a)}$ for all $a \leq b$ in $\mathcal P$.
An equivalent characterization of a finite Eulerian poset
is that each of its nontrivial closed intervals contains 
equal numbers of even and odd rank elements.

If $(\mathcal P, \leq)$ is Eulerian, 
then every interval $[a, b]$ of $\mathcal P$ is Eulerian.
If $\mathcal P$ is Eulerian of rank $d$, 
then the dual poset $\mathcal P^*$ is also an Eulerian poset of rank $d$
with rank function $\rho^*(a) = d - \rho(a)$.

\begin{example}
[{[BBo2, Example 2.3]}]
\label{face poset}
Let $C$ be a $d$-dimensional finite strongly convex polyhedral cone in $\mathbb R^d$.
Then the face poset $\mathcal P_C$ of $C$ satisfies all of the above assumptions 
with minimal element $\{0\}$, maximal element $C$, and is Eulerian of rank $d$.
The rank function $\rho$ is equal to the dimension of the face.
Similarly, 
if $\triangle$ is a full-dimensional polytope in $\mathbb R^{d-1}$, 
then the face poset of $\triangle$, 
regarding both the empty set and $\triangle$ as faces of $\triangle$, 
is Eulerian of rank $d$ with the rank function equal to the dimension of the face plus 1
using the convention $\dim \emptyset =-1$.
\end{example}

\begin{define}
[{[S1, \S 2], [BBo2, Definition 2.4]}]
Let $\mathcal P = [\hat 0, \hat 1]$ be an Eulerian poset of rank $d$.
Define two polynomials $H(\mathcal P, t)$, $G(\mathcal P, t)$ in $\mathbb Z[t]$
by the following recursive rules:
\begin{eqnarray}
G(\mathcal P, t) & = & H(\mathcal P, t) \quad = \quad 1 \qquad \qquad \text{if } d=0,
\nonumber \\
H(\mathcal P, t) & = & \sum_{\hat 0 <a \leq \hat 1} (t-1)^{\rho(a)-1} G([a, \hat 1], t) 
\qquad \qquad \text{for  } d>0
\nonumber \\
G(\mathcal P, t) & = & \tau_{<d/2} ( (1-t) H(\mathcal P, t) ) 
\qquad \qquad \text{for  } d>0
\nonumber
\end{eqnarray}
where $\tau_{<r}$ denotes the truncation operator 
$\mathbb Z[t] \to \mathbb Z[t]$ defined by 
$\tau_{<r} (\sum_i a_i t^i) = \sum_{i<r} a_i t^i$.
\end{define}

\begin{remark}
It is easy to check that if $\mathcal P= \{ \hat 0, \hat 1\}$ is an Eulerian poset of rank $1$,
then one has $G(\mathcal P, t) = H(\mathcal P, t) = 1$.
If $\mathcal P$ is a finite Eulerian poset of rank $2$, 
in which case $\mathcal P = \{ \hat 0, \alpha, \beta, \hat 1\}$
with $\alpha$ and $\beta$ incomparable and $\rho(\alpha) = \rho(\beta)=1$, 
then $H(\mathcal P, t) = 1+t$ and $G(\mathcal P, t) = 1$.
\end{remark}

\begin{theorem}
[{[S1, Theorem 2.4], [BBo2, Theorem 2.5]}]
Let $\mathcal P$ be a finite Eulerian poset of rank $d\geq 1$, 
then $H(\mathcal P, t) = t^{d-1} H(\mathcal P, \frac{1}{t})$.
\end{theorem}

\begin{define}
[{[BBo2, Definition 3.19]}]
Let $\mathcal P$ be a finite Eulerian poset of rank $d \geq 1$.
For $d \geq 2$, 
define $H_{\text{Lef}}(\mathcal P, t)$ to be the polynomial of degree $\leq d-2$  
with the following properties:
\begin{enumerate}
\item
$H_{\text{Lef}}(\mathcal P, t) = t^{d-2} H_{\text{Lef}}(\mathcal P, \frac{1}{t})$, 

\item
$\tau_{\leq \frac{d-2}{2}} H_{\text{Lef}}(\mathcal P, t) = \tau_{\leq \frac{d-2}{2}} H(\mathcal P, t)$.
\end{enumerate}
For $d =1$, put $H_{\text{Lef}}(\mathcal P, t) = 0$.
\end{define}

\begin{lemma}
\label{H_lef}
Let $\mathcal P$ be a finite Eulerian poset of rank $d \geq 1$.
Then 
\begin{eqnarray}
H(\mathcal P, t) - t^{d-1} G(\mathcal P, \frac{1}{t}) \quad = \quad 
&  H_{\text{Lef}}(\mathcal P, t) & \quad =  \quad 
\frac{H(\mathcal P, t) - G(\mathcal P, t)}{t}.
\nonumber
\end{eqnarray}
\end{lemma}

\begin{lemma}
[{[S2, Section 8], [BoM, Lemma 11.2]}]
\label{magical formula}
Let $\mathcal P=[\hat 0, \hat 1]$ be a finite Eulerian poset of rank $d >0$, then
\begin{eqnarray}
\sum_{\hat 0 \leq a \leq \hat 1} G([\hat 0, a], t) (-1)^{\text{rk} [a, \hat 1]} G([a, \hat 1]^*, t) 
& = & 0
\nonumber
\end{eqnarray}
In particular, this implies that the function 
$[a, b] \mapsto G([a, b], t)$ in the incidence algebra $\mathbb I(\mathcal P, \mathbb Z[t])$
is invertible 
and the inverse is given by 
$[a, b] \mapsto (-1)^{\text{rk} [a, b]} G([a, b]^*, t)$.
\end{lemma}

The following is well-known.

\begin{prop}
\label{Ehrhart}
Let $\triangle$ be a full dimensional lattice polytope in $\mathbb R^{d-1}$.
Define the Ehrhart series
$\text{Ehr}_\triangle(t) := \sum_{k \geq 0} l(k \triangle) t^k$.
Then
\begin{enumerate}
\item
$(1-t)^d \text{Ehr}_\triangle(t)$ is a polynomial of degree $\leq d-1$, which we denote by 
\emph{$\mathcal S(\triangle, t)$}.

\item
If $\mathcal T$ is a regular triangulation of $\triangle$ 
meaning that all simplices in $\mathcal T$ have normalized volume $1$, 
and let $a_i$ be the number of $i$-dimensional simplices in $\mathcal T$
which are not contained in the boundary of $\triangle$, 
then 
$\mathcal S(\triangle, t) =  a_{d-1}+ a_{d-2}(t-1) + \cdots + a_0 (t-1)^{d-1}$.

\end{enumerate}
\end{prop}

This concludes the preparation on posets and lattice polytopes.
We now go back to the discussion about the reflexive polytopes 
$\triangle, \triangle^\vee$ in $\mathbb R^n$.
From now on, let $\mathcal P = \mathcal P_\triangle = [\hat 0, \hat 1]$ 
be the face poset of $\triangle$
with minimal element $\hat 0 = \emptyset$, 
maximal element $\hat 1 = \triangle$, 
and rank function $\rho$. 
As pointed out in Example \ref{face poset}, 
$\mathcal P$ is a finite Eulerian poset of rank $n+1$.
There is a one-to-one correspondence 
between the faces of $\triangle$ and the faces of $\triangle^\vee$, 
which induces a canonical isomorphism between the dual poset $\mathcal P^*$ 
and the face poset of $\triangle^\vee$.
As in [BBo2], 
we use the faces of $\triangle$ as indices and denote by 
$\triangle_a$ the face of $\triangle$ corresponding to $a \in \mathcal P$, 
for example $\triangle_{\hat 0} = \emptyset$, 
$\triangle_{\hat 1} = \triangle$, 
and $\rho(a)= \dim \triangle_a +1$.
We denote by $\triangle_a^*$ the face of $\triangle^\vee$ 
that corresponds to the face $\triangle_a$ of $\triangle$.
One has $\dim \triangle_a + \dim \triangle_a^* = n-1$ for all $a \in \mathcal P$.

The cones in the fan $\Sigma_\triangle$ are in one-to-one correspondence with the faces of 
$\triangle$ except $\triangle_{\hat 1} = \triangle$, 
treating the origin as the cone over the empty set.
This allows us to identify $\Sigma_\triangle$ with the set
$[\hat 0, \hat 1) = \mathcal P\backslash \{ \hat 1 \}$.
Denote the orbit of the toric variety $X^\vee = \mathbb P_{\Sigma_{\triangle}}$ 
corresponding to $a \in [\hat 0, \hat 1)$ 
under the cone-orbit correspondence
by $X^\vee_a$, and 
its closure by $\bar X^\vee_a = \cup_{a \leq b <\hat 1} X^\vee_b$.
Also put $Y^\vee_a = Y^\vee \cap X^\vee_a$ and
$\bar Y^\vee_a = Y^\vee \cap \bar X^\vee_a$.
Similarly, we identify $\Sigma_{\triangle^\vee}$ with $(\hat 0, \hat 1]$, 
and denote the orbit of $X = \mathbb P_{\Sigma_{\triangle^\vee}}$
corresponding to $a \in (\hat 0, \hat 1]$ 
by $X_a$, and its closure by $\bar X_a = \cup_{\hat 0 <b \leq a} X_b$.
Put $Y_a = Y \cap X_a$ and 
$\bar Y_a = Y \cap \bar X_a$.

Using the new notations, we have the following decomposition for $\pi^\vee$
\begin{eqnarray}
R\pi^\vee_* \underline{\mathbb Q}_{\widetilde {X^\vee}} & \cong &
\bigoplus_{\hat 0 \leq a < \hat 1} \bigoplus_{k \in \mathbb Z}
\mathcal I_{\bar X^\vee_a}^{\oplus s_{a, k}} [- \rho(a)-k]
\label{decomposition for pi cech}
\end{eqnarray}
and
\begin{eqnarray}
P_{\pi^\vee, b} (t) & = & \sum_{\hat 0 \leq a \leq b} S_a(t) R_{a, b} (t)
\qquad \text{for    }  b \in [\hat 0, \hat 1). 
\label{dCMM 7.3'}
\end{eqnarray}
By [dCMM, Theorem 4.1, Proposition 4.6], 
$P_{\pi^\vee, b}(t)$, 
the Poincar\'{e} polynomial of 
the fiber of $\pi^\vee$ over a point in $X^\vee_b$, 
is equal to $\sum_{\ell \in \mathbb Z} d_\ell (\widetilde{X^\vee}/ b) (t^2-1)^\ell$, 
where $d_\ell(\widetilde{X^\vee}/ b)$ is the number of 
$(\rho(b)-\ell)$-dimensional cones in $\Sigma_{\triangle, \mathcal T}$ 
that are contained in $b \in \Sigma_\triangle$ 
but not in its boundary. 
Let $\mathcal T_b$ be the regular triangulation that $\mathcal T$ induces on $\triangle_b$, 
then clearly $d_\ell(\widetilde{X^\vee}/ b)$ is also the number of 
$(\rho(b)-\ell-1)$-dimensional simplices in $\mathcal T_b$
which do not belong to the boundary of $\triangle_b$.
Proposition \ref{Ehrhart} implies that
$P_{\pi^\vee, b} = \mathcal S(\triangle_b, t^2)$
with the convention
$\mathcal S(\triangle_{\hat 0}, t)= \mathcal S(\emptyset, t) = 1$.
When $\rho(b) \leq 1$, 
i.e. when $\triangle_b$ is the empty set $\emptyset$ or is a vertex of $\triangle$, 
then $\pi^\vee$ maps ${\pi^\vee}^{-1} (X^\vee_b)$ isomorphically to $X^\vee_b$
with fiber a point -
in both cases one has 
$P_{\pi^\vee, b} = \mathcal S(\triangle_b, t^2) = 1$.

It is also well-known that the local intersection cohomology Poincar\'{e} polynomial
$R_{a, b}(t)$ is equal to $G([a, b]^*, t^2)$
(see e.g. [F]).
The equation (\ref{dCMM 7.3'}) now reads
\begin{eqnarray}
\mathcal S (\triangle_b, t^2) & = & 
\sum_{\hat 0 \leq a \leq b} S_a(t) G([a, b]^*, t^2).
\nonumber
\end{eqnarray}
Lemma \ref{magical formula} immediately implies the following

\begin{prop}
\label{S tilde}
$S_a(t)$, defined to be $\sum_{k \in \mathbb Z} s_{a, k} t^{\rho(a)+k}$ in Notation \ref{multiplicity series}, for $a \in [\hat 0, \hat 1)$, 
is equal to 
$$\sum_{\hat 0 \leq c \leq a} \mathcal S(\triangle_c, t^2) (-1)^{\rho(a)-\rho(c)} G([c, a], t^2).
$$
\end{prop}

\begin{remark}
\label{6.14}
The above expression is, in fact, the polynomial
$\widetilde{\mathcal S}(C_a, t^2)$ in [BoM, Definition 5.3],
where $C_a$ is a Gorenstein cone with supporting polyhedron $\triangle_a$
(see [BBo2, Definition 3.3]).
The duality property of $\widetilde{\mathcal S}(C_a, t^2)$ 
in [BoM, Remark 5.4] is, in this case, 
equivalent to the condition 
$s_{a, k} = s_{a, -k}$ in [dCMM, Theorem 5.1(i)]
\end{remark}

\begin{remark}
\label{6.15}
It is clear that $S_{\hat 0}(t)=1$, which means 
$\mathcal I_{X^\vee}$ appears exactly once as a direct summand of 
$R\pi^\vee_* \underline{\mathbb Q}_{\widetilde{X^\vee}}$.
It is also easy to check that
$S_a(t) =0$ if $\rho(a)=1$.
This means that the intersection complexes of the torus-invariant prime Weil divisors of $X^\vee$ make no appearance in the decomposition of $R\pi^\vee_* \underline{\mathbb Q}_{\widetilde{X^\vee}}$.
If $\rho(a)=2$ and $\triangle_a$ is an interval of length $m$, 
then $\mathcal S(\triangle_a, t) = 1+(m-1)t$ and
$S_a(t) = (m-1)t^2$.
\end{remark}

As in the proof of Theorem \ref{5.2}, 
the generating series 
\begin{eqnarray}
\sum_{i \geq 0} \dim (\HH^\bullet(\widetilde{X^\vee}, \mathbb Q)/ \text{ker} [\widetilde{Y^\vee}])_i \, t^i
\label{generating series}
\end{eqnarray}
is equal to 
$$
\sum_{\hat 0 \leq a < \hat 1} \left( \sum_{k \in \mathbb Z} s_{a, k} t^{\rho(a)+k} \right)
\left( \sum_i \dim (\IHH^\bullet(\bar X^\vee_a, \mathbb Q)/\text{ker} [Y^\vee])_i \, t^i \right).
$$
It follows from [F, Theorem 1.1] that
\begin{eqnarray}
\sum_i \dim \IHH^i(\bar X^\vee_a, \mathbb Q) \, t^i & = &  H([a, \hat 1]^*, t^2).
\nonumber
\end{eqnarray}
The hard Lefschetz theorem for intersection cohomology [BBD] implies that
\begin{eqnarray}
\sum_i \dim (\IHH^\bullet(\bar X^\vee_a, \mathbb Q)/\text{ker} [Y^\vee])_i \, t^i & = & 
H_{\text{Lef}}([a, \hat 1]^*, t^2).
\nonumber
\end{eqnarray}
Together with Proposition \ref{S tilde} and Lemma \ref{H_lef}, we have
\begin{eqnarray}
(\ref{generating series}) & = & 
\sum_{a \in [\hat 0, \hat 1)} 
\underbrace{
\left( \sum_{c \in [\hat 0, a]} \mathcal S(\triangle_c, t^2) (-1)^{\rho(a)-\rho(c)} G([c, a], t^2) \right)
}_{S_a(t)}
\underbrace{
\left( \frac{H([a, \hat 1]^*, t^2) - G([a, \hat 1]^*, t^2}{t^2} \right)
}_{H_{\text{Lef}}([a, \hat 1]^*, t^2)}
\nonumber
\end{eqnarray}

\begin{lemma}
\label{6.16}
$\displaystyle{\sum_{a \in [\hat 0, \hat 1)} S_a(t) \, H([a, \hat 1]^*, t^2) = \mathcal S(\triangle, t^2)}$.
\end{lemma}

\begin{proof}
The left-hand side is equal to
$\displaystyle{\sum_{a \in [\hat 0, \hat 1)} 
\left( \sum_{k \in \mathbb Z} s_{a, k} t^{\rho(a)+k} \right) 
\left( \sum_i \dim \mathbb H^i(\bar X^\vee_a, \mathcal I_{\bar X^\vee_a}) \, t^i \right) }$, 
which, by (\ref{decomposition for pi cech}), 
is equal to 
$\displaystyle{\sum_i \dim \mathbb H^i(X^\vee, R\pi^\vee_* \underline{\mathbb Q}_{\widetilde{X^\vee}})  \, t^i  = \sum_i \dim \HH^i(\widetilde{X^\vee}, \mathbb Q) \, t^i }$.
Since $\widetilde{X^\vee} = \mathbb P_{\Sigma_{\triangle, \mathcal T}}$
is smooth, 
the Poincar\'{e} polynomial 
$\displaystyle{ \sum_i \dim \HH^i(\widetilde{X^\vee}, \mathbb Q) \, t^i }$ is equal to 
$\displaystyle{ \sum_i a_i (t^2-1)^{n-i} }$, 
where $a_i$ is the number of $i$-dimensional cones in the fan 
$\Sigma_{\triangle, \mathcal T}$.
Since $\mathcal T$ is a regular projective triangulation of $\triangle$, 
the $a_i$ is precisely the number of $i$-dimensional simplicies in $\mathcal T$
that do not belong to the boundary of $\triangle$.
Hence, by Proposition \ref{Ehrhart}, 
we have $\displaystyle{ \sum_i a_i (t^2-1)^{n-i} = \mathcal S(\triangle, t^2)}$.
\end{proof}

\begin{lemma}
\label{6.17}
$\displaystyle{ \sum_{a \in [\hat 0, \hat 1)} S_a(t) \, G([a, \hat 1]^*, t^2) 
= \sum_{a \in [\hat 0, \hat 1)} \mathcal S(\triangle_a, t^2) (-1)^{n-\rho(a)} 
G([a, \hat 1], t^2)
}$.
\end{lemma}

\begin{proof}
The left-hand side is equal to
\begin{eqnarray}
\sum_{a \in [\hat 0, \hat 1)} 
\left( \sum_{c \in [\hat 0, a]} \mathcal S(\triangle_c, t^2) (-1)^{\rho(a)-\rho(c)} G([c, a], t^2) \right)
G([a, \hat 1]^*, t^2) 
\nonumber
\end{eqnarray}
or 
\begin{eqnarray}
\sum_{c \in [\hat 0, \hat 1)} 
\mathcal S(\triangle_c, t^2) 
\underbrace{\left( 
\sum_{a \in [c, \hat 1)} (-1)^{\rho(a)-\rho(c)} G([c, a], t^2) G([a, \hat 1]^*, t^2)
\right)}_I.
\nonumber 
\end{eqnarray}
By Lemma \ref{magical formula}, 
if we sum over the closed nontrivial interval $[c, \hat 1]$ in expression $I$, 
we get $0$,
hence $I = 0 - (-1)^{n+1-\rho(c)} G([c, \hat 1], t^2) = (-1)^{n-\rho(c)} G([c, \hat 1], t^2)$.
\end{proof}

\begin{prop}
\label{6.18}
One has
\begin{eqnarray}
\sum_{i \geq 0} \dim (\HH^\bullet(\widetilde{X^\vee}, \mathbb Q)/ \text{ker} [\widetilde{Y^\vee}])_i \, t^i
& = & 
\sum_{a \in [\hat 0, \hat 1]} 
t^{-2} (-1)^{n+1-\rho(a)} 
\mathcal S(\triangle_a, t^2) G([a, \hat 1], t^2).
\label{rank formula}
\end{eqnarray}
\end{prop}

\begin{proof}
This is immediate from Lemma \ref{6.16} and Lemma \ref{6.17}
(note that $\triangle_{\hat 1} = \triangle$).
\end{proof}

\begin{remark}
If $C$ is a Gorenstein cone with supporting polyhedron $\triangle$, 
then the right-hand side of (\ref{rank formula}) is 
$\displaystyle{\frac{1}{t^2} \widetilde{\mathcal S}(C, t^2)}$ in the sense of [BoM, Definition 5.3].
\end{remark}

Let $i: Y \hookrightarrow X$ be the inclusion.
Since $Y$ intersects the orbits of $X$ transversely, 
one has canonical isomorphisms
$i^* \mathcal I_X \cong \mathcal I_Y$ and
$i^! \mathcal I_X \cong \mathcal I_Y[-2]$. 
This is used to define 
\begin{eqnarray}
i^* \,\, : \,\, \IHH^k (X, \mathbb Q) & \to & \IHH^k (Y, \mathbb Q)
\nonumber
\end{eqnarray}
and the Gysin map
\begin{eqnarray}
i_! \,\, : \,\, \IHH^k (Y, \mathbb Q) & \to & \IHH^{k+2}(X, \mathbb Q)
\nonumber
\end{eqnarray}
([GM, 5.4.1, 5.4.3]).
The Gysin map is also the dual of 
$i^*: \IHH^{2n-2-k}(X, \mathbb Q) \to \IHH^{2n-2-k}(Y, \mathbb Q)$
with respect to the non-degenerate pairing on intersection cohomology
\begin{eqnarray} 
\IHH^k(X, \mathbb Q) \times \IHH^{2n-k}(X, \mathbb Q) \to \mathbb Q, & & 
\IHH^k(Y, \mathbb Q) \times \IHH^{2n-2-k}(Y, \mathbb Q) \to \mathbb Q.
\nonumber
\end{eqnarray}
The composite 
$i_! i^*: \IHH^\bullet(X, \mathbb Q) \to \IHH^{\bullet+2}(X, \mathbb Q)$
is given by the cup product $\smile\!\! [Y]$.
Since $Y$ is ample, 
the Lefschetz hyperplane theorem for intersection cohomology ([GM])
\footnote{Theorem 7.1 in [GM] is stated for a hyperplane section, 
but the proof works for an ample hypersurface $Y$ 
which is transverse to a Whitney stratification of $X$ by 
noting that the complement $X \backslash Y$ is affine.}
implies that 
$i^*: \IHH^k(X, \mathbb Q) \to \IHH^k(Y, \mathbb Q)$ 
is an isomorphism for $0 \leq k \leq n-2$,
and injective for $k=n-1$.
The hard Lefschetz theorem for intersection cohomology ([BBD]) implies that
\begin{eqnarray}
(\smile\!\! [Y])^k \,\, : \,\,  \IHH^{n-k}(X, \mathbb Q) & \to & \IHH^{n+k}(X, \mathbb Q)
\nonumber
\end{eqnarray}
is an isomorphism for all $k$, $0\leq k \leq n$.
From the weak and hard Lefschetz theorems, 
it is easy to deduce the following facts:

\begin{enumerate}[label=(\roman*)]

\item
$i^*: \IHH^k(X, \mathbb Q) \to \IHH^k(Y, \mathbb Q)$
is surjective for $n \leq k \leq 2n$; 

\item
the Gysin map 
$i_!: \IHH^k(Y, \mathbb Q) \to \IHH^{k+2}(X, \mathbb Q)$ 
is injective for $0\leq k \leq n-2$, 
surjective for $k=n-1$, 
and an isomorphism for $n \leq k \leq 2n-2$; 

\item
$\IHH^k(Y, \mathbb Q) = \text{im } i^* \oplus \text{ker } i_!$ for all $k$, 
and $\text{ker } i^* = \text{ker} (\smile\!\! [Y]: \IHH^\bullet(X, \mathbb Q) \to \IHH^{\bullet+2}(X, \mathbb Q) )$.

\end{enumerate}

It follows from the work of M. Saito ([Sa]) that 
$\IHH^k(X, \mathbb Q)$ and $\IHH^k(Y, \mathbb Q)$ both 
carry a natural pure Hodge structure of weight $k$, 
and the maps 
$i^*: \IHH^k(X, \mathbb Q) \to \IHH^k(Y, \mathbb Q)$ and
$i_!: \IHH^k(Y, \mathbb Q) \to \IHH^{k+2}(X, \mathbb Q) (1)$ 
are morphisms of pure Hodge structures.

\begin{define}
\label{vanishing intersection cohomology}
$\IHH^{n-1}_{\text{van}} (Y, \mathbb Q) := \text{ker } (i_!: \IHH^{n-1}(Y, \mathbb Q) \to \IHH^{n+1}(X, \mathbb Q)(1))$
is called the \emph{vanishing intersection cohomology} of $Y$.
It has a pure Hodge structure of weight $n-1$.
\end{define}

Let $E_{\text{int}}(X; u, v)$ be the Hodge-Deligne polynomial for the intersection cohomology of $X$, likewise one has $E_{\text{int}}(Y; u, v)$.
The reader can refer to [dCMM, Section 3] for a useful discussion about the Hodge-Deligne polynomial of a complex algebraic variety, 
and [BBo2, Section 3] for the intersection cohomology version.
It is known that
$E_{\text{int}}(X; u, v) = H(P, uv)$
([BBo2, Theorem 3.16]).

\begin{define}
\label{vanishing intersection E-polynomial}
We define the $E$-polynomial of the vanishing intersection cohomology of $Y$ as follows
\begin{eqnarray}
E_{\text{int}}^{\text{van}}(Y; u, v) & := & (-1)^{n-1} \sum_{p+q=n-1} h^{p, q}
(\IHH^{n-1}_{\text{van}}(Y, \mathbb Q)) u^p v^q.
\nonumber
\end{eqnarray}
It is clear that $E_{\text{int}}^{\text{van}}(Y; u, v)$
is a homogeneous symmetric polynomial of degree $n-1$.
\end{define}

It follows from the Lefschetz theorems for intersection cohomology that 
\begin{eqnarray}
E_{\text{int}}^{\text{van}}(Y; u, v)  & = &
E_{\text{int}}(Y; u, v) - H_{\text{Lef}}(P, uv),
\nonumber
\end{eqnarray}
([BBo2, Proposition 3.21]), 
and we have the following formula

\begin{prop}
[{[BBo2, Theorem 3.22, (9)]}]
\label{6.22}
\begin{eqnarray}
E_{\text{int}}^{\text{van}}(Y; t, 1) & = & 
\sum_{a \in [\hat 0, \hat 1]} t^{-1} (-1)^{\rho(a)} \mathcal S(\triangle_a, t) G([a, \hat 1], t).
\nonumber
\end{eqnarray}
\end{prop}

We have the following formula for the rank of 
$\smile\!\! [\widetilde{Y^\vee_{\mathcal T}}]:
\HH^\bullet(\widetilde{X^\vee_{\mathcal T}}, \mathbb Q) \to 
\HH^\bullet(\widetilde{X^\vee_{\mathcal T}}, \mathbb Q)$.

\begin{theorem}
\label{general formula}
One has 
\begin{eqnarray}
\dim \HH^\bullet(\widetilde{X^\vee_{\mathcal T}}, \mathbb Q)/\text{ker} [\widetilde{Y^\vee_{\mathcal T}}] & = & 
\dim \IHH^{n-1}_{\text{van}}(Y, \mathbb Q),
\nonumber
\end{eqnarray}
that is, the number of linearly independent period integrals predicted by the HLY hyperplane conjecture for the Calabi-Yau family $\widetilde{Y}$  
is equal to the dimension of the vanishing intersection cohomology of $Y$. 
A more refined formula is 
\begin{eqnarray}
\dim (\HH^\bullet(\widetilde{X^\vee_{\mathcal T}}, \mathbb Q)/\text{ker} [\widetilde{Y^\vee_{\mathcal T}}])_{2k} 
& = &  
h^{k, n-1-k} ( \IHH^{n-1}_{\text{van}}(Y, \mathbb Q) )
\nonumber
\end{eqnarray}
independent of the regular projective triangulation $\mathcal T$ of $\triangle$.
\end{theorem}

\begin{proof}
It follows from Proposition \ref{6.18}, Proposition \ref{6.22}, and Definition \ref{vanishing intersection E-polynomial} that
\begin{eqnarray}
\sum_{k \geq 0} \dim (\HH^\bullet(\widetilde{X^\vee}, \mathbb Q)/ \text{ker} [\widetilde{Y^\vee}])_k \, t^k
& = & 
(-1)^{n-1} E_{\text{int}}^{\text{van}}(Y; t^2, 1) 
\quad = \quad 
\sum_{k=0}^{n-1} 
h^{k, n-1-k} ( \IHH^{n-1}_{\text{van}}(Y, \mathbb Q) ) t^{2k}.
\nonumber
\end{eqnarray}
\end{proof}

\begin{remark}
Let $i^\vee: Y^\vee \to X^\vee$ and $\widetilde{i^\vee}: \widetilde{Y^\vee} \to \widetilde{X^\vee}$
be the inclusions.
The result [M, Theorem 5.6] implies that 
there is a natural decomposition 
$\HH^\bullet(\widetilde{Y^\vee}, \mathbb Q) = \HH^\bullet_{\text{toric}} (\widetilde{Y^\vee}, \mathbb Q)
\oplus \HH^\bullet_{\text{res}}(\widetilde{Y^\vee}, \mathbb Q)$,
where the toric part 
$\HH^\bullet_{\text{toric}} (\widetilde{Y^\vee}, \mathbb Q)$
is the image of 
$\widetilde{i^\vee}^*: \HH^\bullet(\widetilde{X^\vee}, \mathbb Q) \to
\HH^\bullet(\widetilde{Y^\vee}, \mathbb Q)$, 
and the residue part is the image of the residue map
$\text{Res}: \HH^{\bullet+1} (\widetilde{X^\vee}\backslash \widetilde{Y^\vee}, \mathbb Q)
\to \HH^{\bullet}(\widetilde{Y^\vee}, \mathbb Q)$
or the kernel of the Gysin map
$\widetilde{i^\vee}_!: \HH^\bullet(\widetilde{Y^\vee}, \mathbb Q) \to \HH^{\bullet+2}(\widetilde{X^\vee}, \mathbb Q)$.
Furthermore, 
[M, Theorem 5.1] implies that
$\text{ker } \widetilde{i^\vee}^* = \text{ker } (\widetilde{i^\vee}_! \widetilde{i^\vee}^*)
= \text{ker } ( \smile\!\! [\widetilde{Y^\vee}]:
\HH^\bullet(\widetilde{X^\vee}, \mathbb Q) \to 
\HH^\bullet(\widetilde{X^\vee}, \mathbb Q) )$.
Hence, the quotient
$\HH^\bullet(\widetilde{X^\vee}, \mathbb Q)/\text{ker } [\widetilde{Y^\vee}]$
is isomorphic to the toric part $\HH^\bullet_{\text{toric}}(\widetilde{Y^\vee}, \mathbb Q)$ of 
$\HH^\bullet(\widetilde{Y^\vee}, \mathbb Q)$.
On the other hand, one may prove Theorem 5.1 and Theorem 5.6 of [M]
using the decomposition theorem and the Lefschetz theorems for intersection cohomology.
\end{remark}

Note that $Y_a$ is an affine hypersurface in the orbit 
$X_a \cong (\mathbb C^*)^{\rho(a)-1}$ of $X$.
The closure $\bar Y_a$ is an ample hypersurface in the 
$(\rho(a)-1)$-dimensional toric subvariety $\bar X_a$
and is transverse to the toric stratification of $\bar X_a$.
Similar to Definition \ref{vanishing intersection cohomology} and Definition \ref{vanishing intersection E-polynomial}, 
we define the vanishing intersection cohomology
$\IHH_{\text{van}}^{\rho(a)-2}(\bar Y_a, \mathbb Q):=
\text{ker}( (i_a)_!: \IHH^{\rho(a)-2}(\bar Y_a, \mathbb Q) \to \IHH^{\rho(a)}(\bar X_a, \mathbb Q) )$
and its $E$-polynomial 
$E_{\text{int}}^{\text{van}}(\bar Y_a; u, v) := (-1)^{\rho(a)-2}
\sum_{p+q=\rho(a)-2} h^{p, q} (\IHH_{\text{van}}^{\rho(a)-2} (\bar Y_a, \mathbb Q) )
u^p v^q$.
Then one has 
$E_{\text{int}}^{\text{van}}(\bar Y_a; u, v)=
E_{\text{int}}(\bar Y_a; u, v) - H_{\text{Lef}}([\hat 0, a], uv)$
and the following formula

\begin{lemma}
\label{6.25}
For any $a \in (\hat 0, \hat 1]$, 
one has
\begin{eqnarray}
E_{\text{int}}^{\text{van}}( \bar Y_a; t, 1) 
& = & \sum_{c \in [0, \hat a]} t^{-1} (-1)^{\rho(c)} \mathcal S(\triangle_c, t) G([c, a], t).
\nonumber
\end{eqnarray}
\end{lemma}

\begin{proof}
The case $a=\hat 1$ is just Proposition \ref{6.22}.
For the other values of $a$, 
the proof is identical to the proof of [BBo2, Theorem 3.22, (9)].
\end{proof}

\begin{remark}
Let $C$ be the cone in $\mathbb R^{n+1} = \mathbb R^n \times \mathbb R$ over the polyhedron 
$\triangle \times 1$. 
For any $a \in [\hat 0, \hat 1]$, let 
$C_a$ be the face of $C$ which is supported by $\triangle_a$; 
for instance, $C_{\hat 0} = \{0\}$, $C_{\hat 1} = C$.
Let 
\begin{eqnarray}
\widetilde{\mathcal S}(C_a, t) & := & 
\sum_{c \in [\hat 0, a]} \mathcal S(\triangle_c, t) (-1)^{\rho(a)-\rho(c)} G([c, a], t)
\nonumber
\end{eqnarray}
be the $\widetilde{\mathcal S}$-polynomials introduced in [BoM, Definition 5.3].
These polynomials play two roles.
In the range $a \in [\hat 0, \hat 1)$, 
$\widetilde{\mathcal S}(C_a, t^2)$ is equal to 
$S_a(t)=\sum_{k \in \mathbb Z} s_{a, k} t^{\rho(a)+k}$, 
the generating function for the multiplicity of $\mathcal I_{\bar X^\vee_a}[-q]$
in the decomposition of $R\pi^\vee_* \underline{\mathbb Q}_{\widetilde{X^\vee}}$
as $q$ varies
(see Remark \ref{6.14}).
By restricting the decomposition 
(\ref{decomposition for pi cech}) for $\pi^\vee: \widetilde{X^\vee} \to X^\vee$ 
to the Calabi-Yau hypersurfaces,
we obtain the following decomposition for
$\pi^\vee: \widetilde{Y^\vee} \to Y^\vee$:
\begin{eqnarray}
R\pi^\vee_* \underline{\mathbb Q}_{\widetilde {Y^\vee}} & \cong &
\bigoplus_{\hat 0 \leq a < \hat 1} \bigoplus_{k \in \mathbb Z}
\mathcal I_{\bar Y^\vee_a}^{\oplus s_{a, k}} [- \rho(a)-k].
\label{decomposition for pi cech Y}
\end{eqnarray}
So, $S_a(t)$, i.e. $\widetilde{\mathcal S}(C_a, t^2)$, 
is also the generating function for the multiplicity of $\mathcal I_{\bar Y^\vee_a}[-q]$, 
as $q$ varies, 
in the decomposition of $R\pi^\vee_* \underline{\mathbb Q}_{\widetilde{Y^\vee}}$.
In the range $a \in (\hat 0, \hat 1]$, 
Lemma \ref{6.25} implies that
$\widetilde{\mathcal S}(C_a, t)$ has no constant term and 
\begin{eqnarray}
\frac{\widetilde{\mathcal S}(C_a, t)}{t}
& = & 
(-1)^{\rho(a)} E_{\text{int}}^{\text{van}}(\bar Y_a; t, 1) 
\quad = \quad
\sum_{k=0}^{\rho(a)-2} h^{k, \rho(a)-2-k} (\IHH_{\text{van}}^{\rho(a)-2} (\bar Y_a, \mathbb Q) ) t^k.
\nonumber
\end{eqnarray}
In this capacity, 
the duality property 
of $\widetilde{\mathcal S}(C_a, t)$ in 
[BoM, Remark 5.4]
is equivalent to the symmetry of Hodge numbers
$h^{p, q} (\IHH_{\text{van}}^{\rho(a)-2} (\bar Y_a, \mathbb Q)) = 
h^{q, p} (\IHH_{\text{van}}^{\rho(a)-2} (\bar Y_a, \mathbb Q))$
for $p+q = \rho(a)-2$.
When $\rho(a)=1$, 
i.e. $\triangle_a$ is a vertex of $\triangle$,
we have seen in Remark \ref{6.15}
that $\widetilde{\mathcal S}(C_a, t) =0$,
which means that the intersection complexes of torus-invariant prime Weil divisors of $X^\vee$ 
do not appear in the decomposition of 
$R\pi^\vee_* \underline{\mathbb Q}_{\widetilde{X^\vee}}$.
In the other capacity of 
$\widetilde{\mathcal S}(C_a, t)$, 
this is due to the fact that
$Y$ does not pass through the torus-fixed points of $X$, 
i.e. $Y_a = \bar Y_a = \emptyset$ when 
$\triangle_a$ is a vertex of $\triangle$.
When $\rho(a)=2$, 
say $\triangle_a$ is an interval of length $m$, 
then $\widetilde{\mathcal S}(C_a, t)= (m-1)t$, 
$\frac{\widetilde{\mathcal S}(C_a, t)}{t} = (m-1)$.
This corresponds to the fact that 
$\bar Y_a$ is $m$ distinct points in 
$\bar X_a \cong \mathbb P^1$.
For $a=\hat 1$, 
we have 
$\frac{\widetilde{\mathcal S}(C, t)}{t} = 
(-1)^{\dim Y} E_{\text{int}}^{\text{van}}(Y; t, 1)$
and 
$\frac{\widetilde{\mathcal S}(C, t^2)}{t^2}
= \sum_k \dim \HH^k_{\text{toric}}(\widetilde{Y^\vee}, \mathbb Q) t^k$.

Taking the cohomology groups 
on both sides of 
(\ref{decomposition for pi cech Y})
yields
\begin{eqnarray}
\HH^\bullet(\widetilde{Y^\vee}, \mathbb Q) & \cong & 
\bigoplus_{a\in [\hat 0, \hat 1)} \bigoplus_{k \in \mathbb Z}
\IHH^{\bullet-\rho(a)-k} (\bar Y_a^\vee, \mathbb Q)^{\oplus s_{a, k}}, 
\label{cohomological decomposition for pi cech Y}
\end{eqnarray}
where each $\IHH^\bullet(\bar Y_a^\vee, \mathbb Q)$
splits into the direct sum of the ``toric part"
$\text{im} ( (i^\vee_a)^*: 
\IHH^\bullet(\bar X_a^\vee, \mathbb Q) \to \IHH^\bullet(\bar Y_a^\vee, \mathbb Q))$
and the vanishing part
$\IHH^{n-1-\rho(a)}_{\text{van}}(\bar Y^\vee_a, \mathbb Q) = \text{ker} (i^\vee_a)_!$.
Let $C^\vee$ be the dual cone of $C$, 
or the cone over $(\triangle^\vee, 1)$, 
and let $C_a^*$ be the face of $C^\vee$ supported by $\triangle_a^*$.
Then $E_{\text{int}}^{\text{van}}(\bar Y_a^\vee; t, 1)= 
(-1)^{n-1-\rho(a)} \frac{\widetilde{\mathcal S}(C_a^*, t)}{t}$, 
and hence
\begin{eqnarray}
E_{\text{int}}^{\text{van}}(\bar Y^\vee_a; u, v) 
& = & 
u^{n-1-\rho(a)} E_{\text{int}}^{\text{van}}(\bar Y_a^\vee; \frac{v}{u}, 1) 
\quad = \quad 
(-u)^{n+1-\rho(a)} \frac{1}{uv} \widetilde{\mathcal S}(C_a^*, \frac{v}{u})
\nonumber
\end{eqnarray}
for all $a \in [\hat 0, \hat 1)$.
The $E$-polynomial of the toric part 
$(i^\vee_a)^* \IHH^\bullet(\bar X^\vee_a, \mathbb Q)$ of 
$\IHH^\bullet(\bar Y^\vee_a, \mathbb Q)$
is given by 
$H_{\text{Lef}}([a, \hat 1]^*, uv)$.
The string-theoretic $E$-polynomial of $Y^\vee$ is just 
the Hodge-Deligne $E$-polynomial of the crepant resolution $\widetilde{Y^\vee}$.
Using [dCMM, Theorem 4.1], 
the following formula in [BoM, Theorem 7.2]
\begin{eqnarray}
E_{\text{st}}(Y^\vee; u, v) & = & 
\sum_{a \in [\hat 0, \hat 1]} (uv)^{-1} (-u)^{n+1-\rho(a)} 
\widetilde{\mathcal S}(C_a, uv)
\widetilde{\mathcal S}(C_a^*, \frac{v}{u}) 
\quad = \quad 
\frac{\widetilde{\mathcal S}(C, uv)}{uv} + \sum_{a \in [\hat 0, \hat 1)} \cdots
\nonumber \\
& = & 
\sum_{a \in [\hat 0, \hat 1)} \widetilde{\mathcal S}(C_a, uv) H_{\text{Lef}}([a, \hat 1]^*, uv) 
+ \sum_{a \in [\hat 0, \hat 1)} \widetilde{\mathcal S}(C_a, uv) 
E_{\text{int}}^{\text{van}} (\bar Y^\vee_a; u, v)
\nonumber
\end{eqnarray}
could be interpreted as a consequence of the decomposition 
(\ref{decomposition for pi cech Y}).
[BoM, Theorem 8.3] has a similar interpretation in terms of the decomposition theorem.
\end{remark}

\begin{remark}
\label{6.27}
Theorem \ref{5.2} follows as a corollary to Theorem \ref{general formula}.
When $X = \mathbb P^n$, 
$Y$ is a smooth hypersurface of degree $n+1$ in the projective space.
The (vanishing) intersection cohomology of $Y$ is just the ordinary 
(vanishing) cohomology of $Y$, 
hence $h^{k, n-1-k}(\IHH^{n-1}_{\text{van}}(Y, \mathbb Q)) = h^{k, n-1-k} \HH^{n-1}_{\text{van}}(Y, \mathbb Q)$.
Following from works of Griffiths [G], 
the vanishing cohomology of a sufficiently ample hypersurface $Y$ in a smooth variety $X$ can be realized as the residues of meromorphic differential forms with poles along $Y$.
When $Y$ is a smooth hypersurface of degree $d$ in $\mathbb P^n$, 
defined by the equation $f(z_0, z_1, \ldots, z_n) =0$, 
the residue map induces an isomorphism
$$R_f^{pd-n-1} \cong \HH_{\text{van}}^{n-p, p-1}(Y)$$ 
where $R_f^{pd-n-1}$ is the degree $(pd-n-1)$ component of the Jacobian ring
$$R_f = \mathbb C[z_0, z_1, \ldots, z_n]/(\frac{\partial f}{\partial z_0}, \frac{\partial f}{\partial z_1}, \ldots, \frac{\partial f}{\partial z_n})$$
(see [V, Corollary 6.12, Remark 6.8]).
Let $Y$ be the Fermat hypersurface
$z_0^{n+1} + z_1^{n+1} + \cdots + z_n^{n+1}=0$, 
then one has 
$\dim \HH_{\text{van}} ^{n-1-k, k}(Y) = \dim R^{k(n+1)}$
where $R = \mathbb C[z_0, z_1, \ldots, z_n]/(z_0^n, z_1^n, \ldots, z_n^n)$.
The dimension of $R^{k(n+1)}$ is clearly equal to the number 
$a(k(n+1))$ defined in Definition \ref{a(s)}.
\end{remark}

Theorem \ref{general formula} implies that, 
if the hyperplane conjecture is true, 
then the rank of the period sheaf associated to the universal family of smooth Calabi-Yau hypersurfaces in $\widetilde X$ 
is equal to the dimension of the vanishing intersection cohomology 
of a $\triangle$-regular Calabi-Yau hypersurface $Y$ in $X$.
Recall that 
if $Y$ is defined by $f$, then 
the residue of $\frac{\omega}{f}$ defines a cohomology class 
in $\HH^{n-1, 0}(\widetilde Y_f, \mathbb C)$
(see Section 2).
Using the decomposition theorem,
one could show that it in fact 
defines an intersection cohomology class in 
$\IHH^{n-1, 0} (Y_f, \mathbb C)$.
This will be covered in a future paper
when we use the decomposition theorem and the Radon transform
\footnote{This was suggested to the second author by Zhiwei Yun.}
to study the period sheaf of Calabi-Yau hypersurfaces in $\widetilde X$.

\section{Appendix}

Let $\triangle$ be the polytope in $\mathbb R^n$ with vertices
$$(n,-1, \ldots,-1), \,\, (-1, n, \ldots,-1), \,\, \ldots, \,\, (-1, \ldots,-1, n), \,\, (-1,-1, \ldots,-1).$$
We will use repeated \emph{star subdivision} to show that 
$\triangle$ admits a regular projective triangulation.

Recall the definition of star subdivision and its basic properties from [CLS, 11.1].
Let $N$ be a lattice, 
$\Sigma$ be a fan in $N_{\mathbb R}$ 
consisting of strongly convex rational polyhedral cones, 
and let $| \Sigma | = \cup_{\sigma \in \Sigma} \sigma$ be the support of $\Sigma$.
The pair $(N, \Sigma)$ defines a toric variety $X_\Sigma$.
Given a primitive element $v \in |\Sigma| \cap N$, 
let $\Sigma^*(v)$ be the set of the following cones:

(a) $\sigma$, where $v \notin \sigma \in \Sigma$, 

(b) the cone generated by $\tau$ and $v$, where $v \notin \tau \in \Sigma$ and $\{ v \} \cup \tau \subset \sigma \in \Sigma$.

\noindent
The new fan $\Sigma^*(v)$ is called the star subdivision of $\Sigma$ at $v$.
It has the following properties:

(i) $\Sigma^*(v)$ is a refinement of $\Sigma$,

(ii) the induced toric morphism $X_{\Sigma^*(v)} \to X_\Sigma$ is projective.

Using the lattice points in $\partial \triangle$, 
we perform a sequence of star subdivisions on $\Sigma_\triangle$:
$$
\Sigma_\triangle = \Sigma_\triangle^0 \dashrightarrow \Sigma_\triangle^1 \dashrightarrow \Sigma_\triangle^2 \dashrightarrow \cdots
\dashrightarrow \Sigma_\triangle^N.
$$
Since each morphism $X_{\Sigma_\triangle^i} \to X_{\Sigma_\triangle^{i-1}}$ is projective, 
so is their composition $X_{\Sigma_\triangle^N} \to X_{\Sigma_\triangle}$.
Since $X_{\Sigma_\triangle}$ is projective, 
it follows that $X_{\Sigma_\triangle^N}$ is projective.
We also want $X_{\Sigma_\triangle^N}$ to be smooth.
To this end, 
we describe below an order in which to choose the lattice points of $\partial \triangle$ so that the corresponding iterated star subdivisions induce a regular triangulation of $\partial \triangle$.

Let $\triangle_k[m]$ denote $m$ times a $k$-simplex of normalized volume 1.
In this notation, the polytope $\triangle$ is $\triangle_n[n+1]$
with faces
$\triangle_k[n+1]$, $0 \leq k \leq n-1$, 
and a unique interior integral point.
Each face $\triangle_k[n+1]$ contains in its interior a copy of $\triangle_k[n-k]$
with integral vertices.

Consider a simplex $\triangle_k[m]$ with vertices $0, me_1, me_2, \ldots, me_k$, 
where $e_1, \ldots, e_k$ is the standard basis of $\mathbb R^k$, 
and $m \geq k+2$.
Put $v_0 =0$, $v_i = m e_i$ for $1 \leq i \leq k$,
and write $\triangle_k[m] = \langle v_0, v_1, \ldots, v_k \rangle$.
Let $w_0 = e_1 + \cdots + e_k$, and 
$w_i = e_1 + \cdots + e_{i-1} + (m-k) e_i + e_{i+1} + \cdots + e_k$
for $1 \leq i \leq k$.
These are integral points in the interior of $\triangle_k[m]$
whose convex hull is a copy of $\triangle_k[m-k-1]$ 
containing all integral interior points of $\triangle_k[m]$.
We use the sequence $w_0, w_1, \ldots, w_k$ to subdivide $\triangle_k[m]$.
For example, 
$w_0$ divides $\triangle_k[m]$ into $(k+1)$-number of $k$-simplices:
$\langle w_0, v_0, v_1, \ldots, \hat v_i, \ldots, v_k \rangle$, $1 \leq i \leq k$,
and $\langle w_0, v_1, \ldots, v_k \rangle$.
The last one contains $w_1$ as an interior point which further splits into the union of 
$(k+1)$-number of $k$-simplices:
$\langle w_1, v_1, \ldots, v_k \rangle$, 
$\langle w_0, w_1, v_1, v_2, \ldots, \hat v_i, \ldots, v_k \rangle$ for $2 \leq i \leq k$, 
and $\langle w_0, w_1, v_2, \ldots, v_k \rangle$.
Repeat this and we get a subdivision of $\triangle_k[m]$ into 
$(k(k+1)+1)$-number of $k$-simplices:
$$
\langle w_0, \ldots, \hat w_i, \ldots, w_{r-1}, w_r, v_r, \ldots, v_k \rangle, 
$$
$$
\langle w_0, \ldots,  w_r, v_r, v_{r+1}, \ldots, \hat v_j, \ldots, v_k \rangle, 
$$
for $0 \leq i \leq r-1$, $r+1 \leq j \leq k$, $0 \leq r \leq k$, 
and $\triangle_k[m-k-1] = \langle w_0, w_1, \ldots, w_k \rangle$.
Each $$
\langle w_0, \ldots, \hat w_i, \ldots, w_{r-1}, w_r, v_r, \ldots, v_k \rangle
= \langle w_0, \ldots, \hat w_i, \ldots, w_{r-1}, w_r \rangle \ast \langle v_r, \ldots, v_k \rangle
$$
is the join of a face of the smaller simplex $\triangle_k[m-k-1]$
and a face of the bigger one $\triangle_k[m]$.
Moreover, it has the property that its normalized volume equals the product of the normalized volumes of the two faces, i.e. 
$$\text{vol} \langle w_0, \ldots, \hat w_i, \ldots, w_{r-1}, w_r, v_r, \ldots, v_k \rangle
= \text{vol} \langle w_0, \ldots, \hat w_i, \ldots, w_{r-1}, w_r \rangle \cdot 
\text{vol} \langle v_r, \ldots, v_k \rangle.
$$
Hence, a triangulation of $\langle w_0, \ldots, \hat w_i, \ldots, w_{r-1}, w_r, v_r, \ldots, v_k \rangle$ which induces regular triangulations of 
$\langle w_0, \ldots, \hat w_i, \ldots, w_{r-1}, w_r \rangle$ and
$\langle v_r, \ldots, v_k \rangle$
must be regular.
The same is true for the other family of simplices.
Next, we triangulate the interior simplex $\triangle_k[m-k-1]$ as follows.

Let $\triangle_k[l]$, $l \geq 1$, be a simplex with vertices $0, le_1, \ldots, le_k$.
First we use the sequence 
$$(l-1)e_1, \ldots, (l-1) e_k, (l-2) e_1, \ldots, (l-2) e_k, \ldots, e_1, \ldots, e_k$$
to get a triangulation of $\triangle_k[l]$ into $(k(l-1)+1)$-number of $k$-simplices:
$$\langle i e_1, \ldots, i e_j \rangle \ast \langle (i+1) e_j, \ldots, (i+1) e_k \rangle, 
$$
$1 \leq i \leq l-1$, $1\leq j \leq k$, 
and $\langle 0, e_1, \ldots, e_k \rangle$.
Note that
$\langle i e_1, \ldots, i e_j \rangle \ast \langle (i+1) e_j, \ldots, (i+1) e_k \rangle$, 
the join of a face of $\triangle_{k-1}[i]= \langle ie_1, \ldots, ie_k \rangle$ and
a face of $\triangle_{k-1}[i+1] = \langle (i+1) e_1, \ldots, (i+1) e_k \rangle$, 
also satisfies the property that
$$
\text{vol} ( \langle i e_1, \ldots, i e_j \rangle \ast \langle (i+1) e_j, \ldots, (i+1) e_k \rangle )
= \text{vol} \langle i e_1, \ldots, i e_j \rangle \cdot
\text{vol} \langle (i+1) e_j, \ldots, (i+1) e_k \rangle.
$$
Hence, regular triangulations of the $(k-1)$-simplices
$\langle ie_1, \ldots, ie_k \rangle$, $2 \leq i \leq l$
will induce a regular triangulation of $\langle 0, le_1, \ldots, le_k\rangle$.
One proceeds with induction on $k$.

Back to the polytope $\triangle$.
First, we use
the vertices of the simplices $\triangle_{n-1}[1]$ contained in the interiors
of the codim 1 faces $\triangle_{n-1}[n+1]$ for making star subdivisions, 
next we use the vertices of the simplices $\triangle_{n-2}[2]$ 
in the interiors of the codim 2 faces $\triangle_{n-2}[n+1]$
and then the integral points of $\triangle_{n-2}[2]$ 
to carry out subdivisions as described in the previous two paragraphs, 
after that, we move on to codim 3 faces, etc.
In the end, this yields a regular projective triangulation of $\triangle$.

\end{document}